\documentclass[11pt]{amsart}

\usepackage{palatino, mathpazo}
\usepackage[mathscr]{eucal}
\usepackage{amssymb}
\usepackage[all]{xy}

\newtheorem{theorem}{Theorem}[section]
\newtheorem{lemma}[theorem]{Lemma}

\newtheorem{proposition}[theorem]{Proposition}
\newtheorem{corollary}[theorem]{Corollary}

\newtheorem{definition}[theorem]{Definition}
\theoremstyle{remark}
\newtheorem{remark}[theorem]{Remark}
\newtheorem{example}[theorem]{Example}
\numberwithin{equation}{section}

\newcommand{\p}{\partial}
\newcommand{\z}{{\bf z}}
\newcommand{\bz}{{\bf z}}

\begin{document}

\title[Mean field equations and modular forms II]
{Mean field equations, hyperelliptic curves and modular forms: II}
\author{Chang-Shou Lin}
\address{Department of Mathematics and Center for Advanced Studies in Theoretic Sciences (CASTS), National Taiwan University, Taipei}
\email{cslin@math.ntu.edu.tw}

\author{Chin-Lung Wang}
\address{Department of Mathematics and Taida Institute of
Mathematical Sciences (TIMS), National Taiwan University, Taipei}
\email{dragon@math.ntu.edu.tw}
\urladdr{http://www.math.ntu.edu.tw/~dragon/}

\date{September 20, 2016}

\begin{abstract}
A \emph{pre-modular form} $Z_n(\sigma; \tau)$ of weight $\tfrac{1}{2} n(n + 1)$ is introduced for each $n \in \Bbb N$, where $(\sigma, \tau) \in \Bbb C \times \Bbb H$, such that for $E_\tau = \Bbb C/(\Bbb Z + \Bbb Z \tau)$, every non-trivial zero of $Z_n(\sigma; \tau)$, namely $\sigma \not\in E_\tau[2]$, corresponds to a (scaling family of) solution to the mean field equation
\begin{equation} \tag{MFE}
\triangle u + e^u = \rho \, \delta_0
\end{equation}
on the flat torus $E_\tau$ with singular strength $\rho = 8\pi n$. 

In Part I \cite{CLW}, a hyperelliptic curve $\bar X_n(\tau) \subset {\rm Sym}^n E_\tau$, the \emph{Lam\'e curve}, associated to the MFE was constructed. Our construction of $Z_n(\sigma; \tau)$ relies on a detailed study on the correspondence $\Bbb P^1 \leftarrow \bar X_n(\tau) \to E_\tau$ induced from the hyperelliptic projection and the addition map. 

As an application of the explicit form of the weight 10 pre-modular form $Z_4(\sigma; \tau)$, a counting formula for Lam\'e equations of degree $n = 4$ with finite monodromy is given in the appendix (by Y.-C.\ Chou).
\end{abstract}

\maketitle
\small
\tableofcontents
\normalsize
\setcounter{section}{-1}

\section{Introduction} \label{s:intro}
\setcounter{equation}{0}

Let $E = E_\tau = \Bbb C/\Lambda_\tau$, $\tau \in \Bbb H = \{\, \tau \in \Bbb C \mid {\rm Im}\,\tau > 0\,\}$ and $\Lambda = \Lambda_\tau = \Bbb Z \omega_1 + \Bbb Z \omega_2$ with $\omega_1 = 1$ and $\omega_2 = \tau$. In this paper, we continue our study, initiated in \cite{LW, CLW}, on the singular Louville (mean field) equation:
\begin{equation} \label{Liouville-eq}
\triangle u + e^u = 8\pi n\, \delta_0 \quad \mbox{on $E$},
\end{equation}
under the flat metric, with $\delta_0$ being the Dirac measure at $0 \in E$. The characteristic feature of this problem is that its solvability depends on the moduli $\tau$ in a sophisticated manner (even for $n = 1$, cf.~\cite{LW}).  

It was shown in \cite[\S 0.2.5, Theorem 0.3]{CLW} that any solution to \eqref{Liouville-eq} lies in a \emph{scaling family of solutions} $u^\lambda$ through the Liouville formula:
\begin{equation} \label{dev-map}
u^\lambda(z) = \log \frac{8 e^{2\lambda}|f'(z)|^2}{(1 + e^{2\lambda}|f(z)|^2)^2}, \quad \lambda \in \Bbb R,
\end{equation}
where the meromorphic function $f$ on $\Bbb C$ is known as a \emph{developing map} which can be chosen to be \emph{even} and and satisfy the \emph{type II constraints}:
\begin{equation} \label{type2}
f(z + \omega_j) = e^{2i \theta_j} f(z), \quad \theta_j \in \Bbb R, \,\quad j = 1, 2.
\end{equation}
This is also known as the \emph{unitary projective monodromy} condition. 

$f$ has precisely $n$ simple zeros in $E^\times$ characterized by \cite[Theorem 0.6]{CLW}: \smallskip

\emph{Thee $n$ zeros $a_1, \ldots, a_n \in E^\times$ of $f$ satisfy $a_i \ne \pm a_j$ for $i \ne j$, and they are completely determined by the $n - 1$ algebraic equations
\begin{equation} \label{e:P}
\sum_{i = 1}^n \wp'(a_i) \wp^r(a_i) = 0, \qquad r = 0,\ldots, n - 2,
\end{equation}
together with the transcendental equation on Green function
\begin{equation} \label{e:G}
\sum_{i = 1}^n \nabla G(a_i) = 0.
\end{equation}}

Following \cite{CLW}, the affine algebraic curve $X_n \subset {\rm Sym}^n E^\times$ defined by equations \eqref{e:P} and $a_i \ne \pm a_j$ for $i \ne j$ is called the ($n$-th) \emph{Liouville curve}.

We will make use of Weierstrass' elliptic function $\wp(z) = \wp(z; \Lambda)$ and its associated $\zeta$, $\sigma$ functions extensively. We use \cite{Whittaker} as a general reference. 

The Green function on $E$ is defined by $ -\triangle G = \delta_0 - {1}/{|E|}$ and $\int_{E} G = 0$. For $z = x + i y = r\omega_1 + s \omega_2$, $r, s \in \Bbb R$, and $\eta_i = 2\zeta(\tfrac{1}{2}\omega_i)$, $i = 1, 2$, being the quasi-periods, it was shown in \cite[Lemma 2.3, Lemma 7.1]{LW} that
\begin{equation} \label{e:Gz}
-4\pi G_z(z; \tau) = \zeta(z; \tau) - r \eta_1(\tau) - s \eta_2(\tau).
\end{equation}
For $z \in E_\tau[N]$, the $N$ torsion points, this first appeared in \cite{Hecke} where Hecke showed that it is a modular form of weight one with respect to $\Gamma(N) = \{A \in {\rm SL}(2, \Bbb Z) \mid A \equiv I_2 \pmod{N}\}$. Thus we call 
\begin{equation} \label{e:Hecke}
Z(z; \tau) = Z_{r, s}(\tau) := \zeta(r\omega_1 + s\omega_2; \tau) - r \eta_1(\tau) - s \eta_2(\tau),
\end{equation}
$(z, \tau) \in \Bbb C \times \Bbb H$ the \emph{Hecke function}, which is holomorphic \emph{only} in $\tau$. In this paper, analytic functions of this sort are called \emph{pre-modular forms}. 

The notion of pre-modular forms allows us to study \emph{deformations} in $\sigma$ to relate \emph{different} modular forms. Recently this idea was successfully applied in \cite{CKLW} to give a complete solution to \eqref{Liouville-eq} for $n = 1$. In that case \eqref{e:P} is empty and the problem is equivalent to solving \emph{non-trivial zeros} of $Z(z; \tau)$, i.e.~$z \not\in E_\tau[2]$. Thus, a key step towards the general cases is to generalize the pre-modular form $Z = Z_1$ to the corresponding $Z_n$ for all $n \ge 2$.

Our starting point is the \emph{hyperelliptic  geometry} on $X_n$ arising from the integral \emph{Lam\'e equations} on $E_\tau$ \cite[Theorem 0.7]{CLW}:
\begin{equation} \label{lame}
w'' = (n(n + 1) \wp + B)w.
\end{equation}
For $a = (a_1, \ldots, a_n) \in \Bbb C^n$, let $w_a(z)$ be the classical \emph{Hermite--Halphen ansatz}:
\begin{equation} \label{ansatz}
w_a(z) := e^{z\sum \zeta(a_i; \tau)} \prod_{i = 1}^n \frac{\sigma(z - a_i; \tau)}{\sigma(z; \tau)}.
\end{equation} 
\smallskip

\emph{Denote $[a] := a \pmod{\Lambda}$. Then $[a] \in X_n$ if and only if $w_a$ and $w_{-a}$ are independent solutions to \eqref{lame}. In that case, the parameter $B$ equals
\begin{equation} \label{e:Ba}
B_a := (2n - 1)\sum_{i = 1}^n \wp(a_i).
\end{equation}}

\emph{The compactified curve $\bar X_n \subset {\rm Sym}^n E$ is a hyperelliptic curve, known as the \emph{Lam\'e curve}, with the addd points $\bar X_n \setminus X_n$ being the branch points of the hyperelliptic projection $B: \bar X_n \to \Bbb P^1$. The \emph{point at infinity} $0^n \in \bar X_n$ is always smooth. The finite branch points satisfy $a \in (E^\times)^n$, $a_i \ne a_j$ for $i \ne j$, and $\{a_1, \cdots, a_n\} = \{-a_1, \cdots, -a_n\}$; $w_a = w_{-a}$ is still a solution to \eqref{lame} with $B= B_a$. These solutions are known as the \emph{Lam\'e functions}.} 

\emph{Let $Y_n = B^{-1}(\Bbb C)$ be the finite part of $\bar X_n$. $Y_n$ can be parametrized by 
$$
Y_n \cong \{(B, C) \mid C^2 = \ell_n(B)\}
$$ 
where $\ell_n(B)$ is the Lam\'e polynomial in $B$ of degree $2n + 1$. $\bar X_n$ is smooth if and only if $\ell_n(B)$ has no multiple roots.} \smallskip

Further technical details needed from \cite[Theorem 0.7]{CLW} are summarized in Proposition \ref{typeII} and Theorem \ref{hyper-thm}. 

By the anti-symmetry of $\nabla G$, \eqref{e:G} holds automatically on the branch points of $Y_n$, hence they are referred as \emph{trivial solutions}. We will construct a pre-modular form $Z_n(\sigma; \tau)$ with $\sigma \in E_\tau$ which is naturally associated to the family of hyperelliptic curves $\bar X_n(\tau)$, $\tau \in \Bbb H$. The goal is to show that any non-trivial solution $a = \{a_1, \cdots, a_n\} \in X_n$ to  \eqref{e:G} comes from the zero of $Z_n(\sigma; \tau)$ with $\sigma = \sum_{i = 1}^n a_i \not\in E_\tau[2]$, and vice versa. \smallskip

Consider the meromorphic function 
\begin{equation} \label{def-zn}
\z_n(a) := \zeta \Big(\sum_{i = 1}^n a_i \Big) - \sum_{i = 1}^n \zeta(a_i)
\end{equation}
on $E^n$. If $\sum_{i = 1}^n a_i \ne 0$ then
\begin{equation*}
\begin{split}
-4\pi \sum \nabla G(a_i)  = \sum (\zeta(r_i \omega_1 + s_i \omega_2) - r_i \eta_1 - s_i \eta_2) = Z(\sum a_i) - \z_n(a).
\end{split}
\end{equation*}
Hence the Green function equation \eqref{e:G} is equivalent to
\begin{equation} \label{y=Z}
\z_n(a) = Z\Big(\sum_{i = 1}^n a_i\Big).
\end{equation}
This motivates us to study the map
\begin{equation} \label{e:sigma}
\sigma_n: \bar X_n \to E, \qquad a \mapsto \sigma_n(a) := \sum_{i = 1}^n a_i
\end{equation}
induced from the addition map $E^n \to E$. The algebraic curve $\bar X_n(\tau)$ might be singular for some $\tau$, but it must be irreducible (c.f.~Theorem \ref{hyper-thm} (3)). In particular, $\sigma_n$ is a finite morphism and $\deg \sigma_n$ is defined. 

Recall that a \emph{node} is a singularity of the simplest analytic type $y^2 = x^2$. 

\begin{theorem} [$=$ Theorem \ref{t:nodal} + Theorem \ref{degree}] \label{deg-sigma}
The Lam\'e curve $\bar X_n$ has at most nodal singularities. Moreover, the map $\sigma_n: \bar X_n \to E$ has degree $\tfrac{1}{2} n(n + 1)$.
\end{theorem}

From Theorem \ref{deg-sigma}, there is a polynomial 
$$
W_n(\bz) \in \Bbb Q[g_2, g_3, \wp(\sigma), \wp'(\sigma)][\bz]
$$ 
of degree $\tfrac{1}{2} n(n + 1)$ in $\bz$ which defines the (branched) covering map $\sigma_n$. Throughout the paper we use $\sigma$ as the coordinate on $E$ in $\sigma_n: \bar X_n \to E$ and this should not be confused with the Weierstrass $\sigma$ function.

The next task is to find a natural primitive element of this covering map, namely a rational function on $\bar X_n$ which has $W_n$ as its minimal polynomial. This is achieved by the following fundamental theorem:

\begin{theorem} \label{t:primitive}
The rational function $\z_n \in K(\bar X_n)$ is a primitive generator for the field extension $K(\bar X_n)$ over $K(E)$ which is integral over the affine curve $E^\times$. 
%
\end{theorem}

This means that $W_n(\z_n) = 0$, and conversely for general $\tau$ and $\sigma = \sigma_0 \in E_\tau$, the roots of $W_n(\bz)(\sigma_0; \tau) = 0$ are precisely those $\tfrac{1}{2} n(n + 1)$ values $\bz = \z_n(a)$ with $\sigma_n(a) = \sigma_0$. The proof is given in \S \ref{primitive}, Theorem \ref{main-thm}. 

A major tool used is the \emph{tensor product} of two Lam\'e equations $w'' = I_1 w$ and $w' = I_2 w$, where $I = n(n + 1) \wp(z)$, $I_1 = I + B_a$ and $I_2 = I + B_b$. For a general point $\sigma_0 \in E$, we need to show that the $\tfrac{1}{2} n (n + 1)$ points on the fiber of $\bar X_n \to E$ above $\sigma_0$ has distinct $\z_n$ values. From \eqref{def-zn}, it is enough to show that for $\sigma_n(a) = \sigma_n(b) = \sigma_0$, $\sum \zeta(a_i) = \sum \zeta(b_i)$ implies $B_a = B_b$. Since then $a = b$ if $\sigma_0 \not\in E[2]$. 

If $w_1'' = I_1 w_1$ and $w_2'' = I_2 w_2$, then the product $q = w_1 w_2$ satisfies the fourth order ODE (tensor product) given by
\begin{equation} \label{intro-4th-ode}
q'''' -2(I_1 + I_2) q'' -6I'q' + ((B_a - B_b)^2 - 2I'')q = 0.
\end{equation}
We remark that if $B_a = B_b$, then $I_1 = I_2$ and $q$ actually satisfies a third order ODE as the \emph{second symmetric product} of a Lam\'e equation. This is a useful tool in Part I \cite{CLW} in the study of the Lam\'e curve. 

If however $a \ne b$, by (\ref{ansatz}) and addition law, $q = w_a w_{-b} + w_{-a} w_b$ is an \emph{even elliptic function} solution to (\ref{intro-4th-ode}), namely a \emph{polynomial} in $x = \wp(z)$. This leads to strong constraints on (\ref{intro-4th-ode}) in variable $x$ and eventually leads to a contradiction for generic choices of $\sigma_0$.

Now we set (cf.~Corollary \ref{c:pre-mod})
\begin{equation} \label{new-pre-mod}
Z_n(\sigma; \tau) := W_n(Z)(\sigma; \tau).
\end{equation}
Then $Z_n(\sigma;\tau)$ is pre-modular of weight $\tfrac{1}{2} n(n + 1)$. From the construction and (\ref{y=Z}) it is readily seen that $Z_n(\sigma; \tau)$ is the generalization of the Hecke function we are looking for. In fact, for $n \ge 1$, we have 

\begin{theorem} \label{t:corr}
Solutions to the singular Liouville equation \eqref{Liouville-eq} correspond to zeros of pre-modular form $Z_n(\sigma; \tau)$ in \eqref{new-pre-mod} with $\sigma \not\in E_\tau[2]$.
\end{theorem}

We will also present a version of Theorem \ref{t:corr} in terms of \emph{monodromy groups} of Lam\'e equations (cf.~Theorem \ref{t:lame-mono}). 

For $\sigma \in E_\tau[N]$, the $N$-torsion points, the modular form $Z_2(\sigma; \tau)$ and $Z_3(\sigma; \tau)$ were first constructed by Dahmen \cite{Dahmen} in his study on integral Lam\'e equations (\ref{lame}) with algebraic solutions (i.e.~with finite monodromy group). For $n \ge 4$, the existence of a modular form $Z_n(\sigma; \tau)$ of weight $\tfrac{1}{2} n(n + 1)$ was also conjectured in \cite{Dahmen}. This is now settled by our results. 

It remains to find effective and explicit constructions of $Z_n$. Since $\sigma$ is defined by the addition map, which is purely algebraic, in principle this allows us to compute the polynomial $W_n(\bz)$ for any $n \in \Bbb N$ by eliminating variables $B$ and $C$, though in practice the needed calculations are very demanding and time consuming. 

In a different direction, the Lam\'e curve had also been studied extensively in the \emph{finite band integration theory}. In the complex case, this theory concerns about the eigenvalue problem on a second order ODE $L w := w'' - I w = Bw$ with eigenvalue $B$. The potential $I = I(z)$ is called a \emph{finite-gap (band) potential} if the ODE has only logarithmic free solutions except for a finite number of $B \in \Bbb C$. The integral Lam\'e equations (with $I(z) = n(n + 1)\wp(z)$) provide good (indeed earliest) examples of them. Using this theory, Maier \cite{Maier} had recently written down \emph{an explicit map} $\pi_n: \bar X_n \to E$ in terms of coordinate $(B, C)$ on $\bar X_n$ (in our notations). It turns out we can prove
\begin{theorem}[c.f.~Theorem \ref{pi=sigma}]
The map $\pi_n$ agrees with $\sigma_n: \bar X_n \to E$.
\end{theorem}

This provides an alternative way to compute $W_n(\bz)$ by eliminating $B$, $C$, and \S \ref{resultant} is devoted to this explicit construction. In particular the weight 10 pre-modular form $Z_4(\sigma; \tau)$ is explicitly written down (c.f.~Example \ref{ex:W4}).

The existence and effective construction of $Z_n(\sigma; \tau)$ opens the door to extend our complete results on \eqref{Liouville-eq} for $n = 1$ (established in \cite{LW, LW4, CKLW}) to general $n \in \Bbb N$. As a related application, the explicit expression of $Z_4$ is used to solve Dahmen's conjecture on a counting formula for Lam\'e equations \eqref{lame} with finite monodromy for $n = 4$. The method works for general $n$ once $Z_n$ is shown to have expected asymptotic behavior at cusps. The details is written by Y.-C.\ Chou and is included in Appendix \ref{Chou}.

\section{Geometry of $\sigma_n: \bar X_n \to E$} \label{geom-Xn}
\setcounter{equation}{0}

The aim of this section is to prove Theorem \ref{deg-sigma}. We first review and extend some technical details on results from \cite{CLW} quoted in \S \ref{s:intro}. 

\begin{proposition} \cite[Theorem 6.5]{CLW} \label{typeII}{\ }
Let $a_1, \cdots, a_n$ be the zeros of a developing map $f$ for equation \eqref{Liouville-eq}. Then the logarithmic derivative $g = f'/f$ is given by 
\begin{equation} \label{g=logf}
\begin{split}
g(z) = \sum_{i = 1}^n \frac{\wp'(a_i)}{\wp(z) - \wp(a_i)}.
\end{split}
\end{equation}
Moreover, $g(z)$ has ${\rm ord}_{z = 0}\, g(z) = 2n$, and $a_i \not\in E[2]$, $a_i \ne \pm a_j$ for $i \ne j$.
\end{proposition}

\emph{The condition ${\rm ord}_{z = 0}\, g(z) = 2n$ leads to the $n - 1$ equations for $a_1, \ldots, a_n$ given in \eqref{e:P}}: Under the notations $(w, x_j, y_j) = (\wp(z), \wp(p_j),
\wp'(p_j))$,
\begin{equation*}
\begin{split}
g(z) &= \sum_{j = 1}^n \frac{1}{w} \frac{y_j}{1 - x_j/w} \\
&= \sum_{j = 1}^n \frac{y_j}{w} + \sum_{j = 1}^n \frac{y_j
x_j}{w^2} + \cdots + \sum_{j = 1}^n \frac{y_j x_j^r}{w^{r + 1}} +
\cdots.
\end{split}
\end{equation*}
Since $g(z)$ has a zero at $z = 0$ of order $2n$ and $1/w$ has a
zero at $z = 0$ of order two, we get $x_i \ne x_j$ for $i \ne j$ and
\begin{equation} \label{poly-sys}
\sum y_i x_i^r = 0, \qquad r = 0, \ldots, n - 2.
\end{equation} 
This, together with the Weierstrass equation $y_i^2 = 4 x_i^3 - g_2 x_i - g_3$, gives the polynomial system describing the developing maps. 

\emph{The Green equation \eqref{e:G} is equivalent to the type II condition \eqref{type2}}: The argument is essentially contained in \cite[Lemma 2.4]{LW2}.
By the addition law,
\begin{equation*}
\begin{split}
f &= \exp \int g\,dz \\
&= \exp \int \sum_{i = 1}^n (2 \zeta(a_i) -
\zeta(a_i - z) -
\zeta(a_i + z))\,dz\\
&= e^{2\sum_{i = 1}^n \zeta(a_i)z} \prod_{i = 1}^n \frac{\sigma(z
- a_i)}{\sigma(z + a_i)}.
\end{split}
\end{equation*}
We then calculate the monodromy effect on $f$ from 
\begin{equation} \label{e:sigma-trans}
\sigma(z + \omega_j) = - e^{\frac{1}{2} \eta_i (z + \frac{1}{2} \omega_j)} \sigma(z), \qquad j = 1, 2.
\end{equation}
Let $a_i = r_i \omega_1 + s_i \omega_2$ for $i = 1, \ldots, n$. By way of the Legendre relation $\eta_1 \omega_2 - \eta_2 \omega_1 = 2\pi i$ we compute easily that 
\begin{equation} \label{type-II}
\begin{split}
f(z + \omega_1) &= e^{-4\pi i\sum_i s_i + 2\omega_1(\sum\zeta(a_i) - \sum r_i \eta_1 - \sum s_i \eta_2)} f(z), \\
f(z + \omega_2) &= e^{4\pi i \sum_i r_i + 2\omega_2(\sum\zeta(a_i) - \sum r_i \eta_1 - \sum s_i \eta_2)} f(z).
\end{split}
\end{equation}
By \eqref{e:Gz}, the equivalence of \eqref{e:G} and \eqref{type2} follows immediately. 
\smallskip

The Liouville curve $X_n \subset {\rm Sym}^n E$, defined by \eqref{poly-sys} or \eqref{e:P}, has a hyperelliptic structure under the 2 to 1 map $X_n \to \Bbb P^1$, $(x_i, y_i)_{i = 1}^n \mapsto (2n - 1) \sum_{i = 1}^n x_i$ as in \eqref{e:Ba}. This is closely related to the integral Lam\'e equations \eqref{lame} since $f = w_a/w_{-a}$ where $w_a$ is the ansatz solution \eqref{ansatz}.

The full information on the compactified hyperelliptic curve $\bar X_n \to \Bbb P^1$, the Lam\'e curve, especially on the branch points added, is described by 

\begin{theorem} \cite[Theorem 0.7]{CLW} \label{hyper-thm} {\ }
\begin{itemize}
\item[(1)] 
The natural compactification $\bar X_n \subset {\rm Sym}^n E$ coincides with the, possibly singular, projective model of the hyperelliptic curve defined by
\begin{equation}
\begin{split}
C^2 &= \ell_n(B, g_2, g_3) \\
&= 4B s_n^2 + 4g_3 s_{n - 2} s_n - g_2 s_{n - 1} s_n - g_3 s_{n - 1}^2,
\end{split}
\end{equation}
in $(B, C)$, where $s_k = s_k(B, g_2, g_3) = r_k B^k + \cdots \in \Bbb Q[B, g_2, g_3]$, is a recursively defined polynomial of homogeneous degree $k$ with $\deg g_2 = 2$, $\deg g_3 = 3$, and $B = (2n - 1) s_1$. 

\item[(2)]
$\deg \ell_n = 2n + 1$ and $\bar X_n$ has arithmetic genus $g = n$. 

\item[(3)]
 The curve $\bar X_n$ is smooth except for a finite number of $\tau$, namely the discriminant loci of $\ell_n(B, g_2, g_3)$ so that $\ell_n(B)$ has multiple roots. $\bar X_n$ is an irreducible curve which is smooth at infinity.

\item[(4)] 
The $2n + 2$ branch points $a \in \bar X_n \setminus X_n$ are characterized by $-a = a$. In fact $\{-a_i\} \cap \{a_i\} \ne \emptyset \Rightarrow -a = a$. Also $0 \in \{a_i\} \Rightarrow a = (0, 0, \cdots, 0)$.

\item[(5)]
The limiting system of (\ref{poly-sys}) at $a = 0^n$ is given by
\begin{equation} \label{lim-eq}
\sum_{i = 1}^n t_i^{2r + 1} = 0, \quad r = 1, \ldots, n - 1
\end{equation} 
under the non-degenerate constraints
$t_i \ne 0$, $t_i \ne -t_j$.
Moreover, \eqref{lim-eq} has a unique non-degenerate solution in $\Bbb P^{n - 1}$ up to permutations. It gives the tangent direction $[t] \in \Bbb P (T_{0^n}(\bar X_n)) \subset \Bbb P(T_{0^n}({\rm Sym}^n E))$.

\item[(6)]
In terms of $a \in Y_n$, $(B, C)$ can be parameterized by $B(a) = B_a$ and
\begin{equation} \label{e:C(a)}
C(a) = \wp'(a_i) \prod_{j \ne i} (\wp(a_i) - \wp(a_j)), \quad \mbox{for any $i = 1, \ldots, n$}.
\end{equation} 
\end{itemize}
\end{theorem}

The smooth point $a = 0^n \in \bar X_n$ is referred as \emph{the point at infinity}. For the other $2n + 1$ \emph{finite branch points} with $a = -a$, the ansatz solution \eqref{ansatz} $w_a = w_{-a}$ is still a solution to the Lam\'e equation. In the literature, these $2n + 1$ functions are known as the \emph{Lam\'e functions}. 

Notice that \eqref{e:C(a)} arises from \eqref{g=logf} and ${\rm ord}_{z = 0}\, g_a(z) = 2n$ in
\begin{equation*}
g_a(z) := \sum_{i = 1}^n \frac{\wp'(a_i)}{\wp(z) - \wp(a_i)} = \frac{\sum_{i = 1}^n \wp'(a_i) \prod_{j \ne i} (\wp(z) - \wp(a_j))}{\prod_{i = 1}^n (\wp(z) - \wp(a_i))},
\end{equation*}
where the numerator reduces to the constant $C(a)$. By working with \eqref{e:C(a)}, we may say a little more on the possible singularities of $\bar X_n(\tau)$:

\begin{theorem} \label{t:nodal}
$\bar X_n$ has at most nodal singularities. That is, $\ell_n(B)$ has at most double roots. At such a point $a \in Y_n \setminus X_n$, both local branches are smooth and $C$ could be used as a local coordinate.
\end{theorem}

\begin{proof}
Denote by $b = \{b_1, \cdots, b_n\} \in X_n$ a point near the branch point $a = \{a_1, \cdots, a_n\} \in Y_n \setminus X_n$. \eqref{e:C(a)} implies that, for $a_i = -a_i$ (2-torsion) in $E$,
\begin{equation*} 
C(b) = \Big[\wp''(a_i) \prod\nolimits_{j \ne i} (\wp(a_i) - \wp(a_j)) \Big] (b_i - a_i) + o(|b_i - a_i|).
\end{equation*}
Since $\wp''(a_i) \ne 0$, $C$ can be used as a parameter and $b_i'(0) \ne 0, \infty$. 

Similarly for $a_i$ not a 2-torsion point, we denote by $a_{i'} = -a_i$ and get
$$
C(b) = \Big[\wp'(a_i)^2 \prod\nolimits_{j \ne i, i'} (\wp(a_i) - \wp(a_j)) \Big] (b_i + b_{i'}) + o(|b_1 + b_2|).
$$
Since $\wp'(a_i) \ne 0$, $C$ can be used as a parameter and $b_i'(0) + b_{i'}'(0) \ne 0, \infty$. Again, from \eqref{e:C(a)} we deduce that $b_{i'}(C) = -b_i(-C)$. So $b_{i'}'(0) = b_i'(0)$ and hence they are neither $0$ nor $\infty$. 

In summary, the paramaterization $C \mapsto b(C)$ is well defined, holomorphic and non-degenerate in any chosen branch of $Y_n$ near $a = b(0)$. Since the analytic structure at $a \in Y_n$ is of the form $C^2 = (B - \lambda)^m$, this is possible if and only if $m = 1, 2$. The singular case corresponds to $m = 2$ which leads to a double point. The two branches are all non-singular at $a$.
\end{proof}

There are four species of Lame functions, depending on the number of half periods contained in $\{a_i\}$. We call them being of type O, I, II, and III respectively. For $n = 2k$ being even, $a$ must be of type O or II. For $n = 2k + 1$ being odd, $a$ must be of type I or III. There are factorizations of the polynomial $\ell_n(B)$ according to the types:

\begin{proposition} \cite{Halphen, Whittaker} \label{4lame}
In terms of $e_i = \wp(\tfrac{1}{2}\omega_i)$, we may write
$$
\ell_n(B) = c_n^2 l_0(B) l_1(B) l_2(B) l_3(B),
$$ 
where $c_n \in \Bbb Q^+$ is a constant, $l_i(B)$'s are monic polynomials in $B$ such that
\begin{itemize}
\item[(1)]
For $n = 2k$, $l_0(B)$ consists of type O roots with $\deg l_0(B) = \tfrac{1}{2} n + 1 = k + 1$. For $i = 1, 2, 3$, $l_i(B)$ consists of type II roots $a$ which does not contain $\tfrac{1}{2} \omega_i$. Moreover, $\deg l_i(B) = \tfrac{1}{2} n = k$.
\item[(2)]
For $n = 2k + 1$, $l_0(B)$ consists of type III roots with $\deg l_0(B) = \tfrac{1}{2}(n - 1) = k$. For $i = 1, 2, 3$, $l_i(B)$ consists of type I roots $a$ which contains $\tfrac{1}{2} \omega_i$. Moreover, $\deg l_i(B) = \tfrac{1}{2}(n + 1) = k + 1$.
\end{itemize}
\end{proposition}

We remark that Proposition \ref{4lame}, Theorem \ref{hyper-thm} (4), (5) and Theorem \ref{t:nodal} will be used in the proof of Theorem \ref{deg-sigma} ($=$ Theorem \ref{degree} later in this section). Here are some examples to illustrate Proposition \ref{4lame}:

\begin{example} \label{lame-curve}
Decomposition $\ell_n(B) = c_n^2 l_0(B) l_1(B) l_2(B) l_3(B)$ for $1 \le n \le 5$.

(1) $n = 1$, $k = 0$, $\bar X_1 \cong E$, $C^2 = \ell_1(B) = 4 B^3 - g_2 B - g_3 = 4\prod_{i = 1}^3(B - e_i)$.

(2) $n = 2$, $k = 1$, (notice that $e_1 + e_2 + e_3 = 0$) \small
\begin{equation*}
\begin{split} 
C^2 = \ell_2(B) &= \tfrac{4}{81} B^5 - \tfrac{7}{27} g_2 B^3 + \tfrac{1}{3} g_3 B^2 + \tfrac{1}{3} g_2^2 B - g_2 g_3 \\
&= \frac{2^2}{3^4}(B^2 - 3g_2) \prod_{i = 1}^3 (B + 3 e_i).
\end{split}
\end{equation*} \normalsize

(3) $n = 3$, $k = 1$, $\deg l_i(B) = 2$ for $i = 1, 2, 3$, \small
\begin{equation*}
\begin{split}
C^2 = \ell_3(B) &= \frac{1}{2^2 3^4 5^4} B(16B^6 - 504 g_2 B^4 + 2376 g_3 B^3 \\
&\qquad + 4185 g_2^2 B^2 - 36450 g_2 g_3 B + 91125 g_3^2 - 3375 g_2^3) \\
&=  \frac{2^2}{3^4 5^4} B \prod_{i = 1}^3 (B^2 - 6e_i B + 15(3 e_i^2 - g_2)).
\end{split}
\end{equation*} \normalsize

(4) $n = 4$, $k = 2$, $\deg l_0(B) = 3$, \small
\begin{equation*}
\begin{split}
C^2 = \ell_4(B) &= \frac{1}{3^8 5^4 7^4} (B^3 - 52 g_2 B + 560 g_3) \prod_{i = 1}^3 (B^2 + 10 e_i B - 7(5e_i^2 + g_2)).
\end{split}
\end{equation*} \normalsize

(5) $n = 5$, $k = 2$, $\deg l_i(B) = 3$ for $i = 1, 2, 3$, \small
\begin{equation*}
\begin{split}
C^2 = \ell_5(B) &= \frac{1}{3^{12} 5^4 7^4 11^2} (B^2 - 27g_2) \\
&\,\times \prod_{i = 1}^3 (B^3 - 15 e_i B^2 + (315 e_i^2 - 132 g_2)B + e_i(2835 e_i^2 - 540 g_2 )).
\end{split}
\end{equation*} \normalsize
\end{example}

We are now ready to study the addition map $\sigma_n: \bar X_n \to E$, $a \mapsto \sigma_n(a) = \sum_{i = 1}^n a_i$ defined in \eqref{e:sigma}. In the rest of this section we determine $\deg \sigma_n$. 

For the reader's convenience we recall some definitions and facts. The function field $K(C)$ is defined for any irreducible algebraic curve $C$. For a finite morphism of irreducible curves $f: X \to Y$, $K(X)$ is a finite extension of $K(Y)$ and the degree of $f$ is defined by $\deg f = [K(X): K(Y)]$. Geometrically $\deg f$ is also the number of points for a general fiber $f^{-1}(p)$, $p \in Y$. A standard reference is \cite[II.6, Proposition 6.9]{rH}, where nonsingular curves are treated. The irreducible case is reduced to the nonsingular case through normalizations $\tilde X \to X$ and $\tilde Y \to Y$, since it is clear that the induced finite morphism $\tilde f: \tilde X \to \tilde Y$ has the same degree as $f$. Furthermore, the definition also extends to the case $f: X \to Y$ where $X = \bigcup_{i = 1}^k X_i$ has a finite number of irreducible components. We require that $f|_{X_i}$ is a finite morphism for each $i$ and then $\deg f := \sum_{i = 1}^k \deg f|_{X_i}$. Since all curves considered here are proper (projective), it is enough to require $f|_{X_i}$ to be non-constant to ensure that it is a finite morphism.

\begin{theorem} \label{degree}
The map $\sigma_n: \bar X_n \to E$ has degree $\tfrac{1}{2} n (n + 1)$. 
\end{theorem}

\begin{proof} 
The idea is to apply \emph{Theorem of the Cube} \cite[p.58, Corollary 2]{Mumford} for morphisms from an arbitrary variety $V$ (not necessarily smooth) into abelian varieties (here the torus $E$): For \emph{any} three morphisms $f, g, h: V \to E$ and a line bundle $L \in {\rm Pic}\,E$, we have
\begin{equation} \label{t:cube}
\begin{split}
(f + g + h)^*L &\cong (f + g)^*L \otimes (g + h)^*L \otimes (h + f)^*L \\
&\qquad \otimes f^*L^{-1} \otimes g^*L^{-1} \otimes h^*L^{-1}.
\end{split}
\end{equation}
We will apply it to the algebraic curve $V = V_n \subset E^n$ which consists of the ordered $n$-tuples $a$'s so that $V_n/S_n = \bar X_n$.

For any line bundle $L$ and any \emph{finite morphism} $f: V \to E$, we have $\deg f^* L = \deg f \deg L$. In the following we fix an $L$ with $\deg L = 1$.

We prove inductively that for $j = 1, \ldots, n$ the morphism $f_j: V_n \to E$ defined by
$$
f_j(a) := a_1 + \cdots + a_j
$$
has $\deg f_j^* L = \tfrac{1}{2}j(j + 1) n!$. The case $j = n$ then gives the result since $f_n$ is a finite morphism which descends to $\sigma_n$ under the $S_n$ action. (Notice that the map $f_j$ can not descend to a map on $\bar X_n$ for all $j < n$.) 

Assuming first that it has been proved for $j = 1, 2$. To go from $j$ to $j + 1$, we let $f(a) = f_{j - 1}(a)$, $g(a) = a_{j}$, and $h(a) = a_{j + 1}$. Then by (\ref{t:cube}), $f_{j + 1}^*L$ has degree $n!$ times
$$
\tfrac{1}{2}j(j + 1) + 3 + \tfrac{1}{2} j (j + 1) - \tfrac{1}{2} (j - 1)j - 1 - 1 = \tfrac{1}{2} (j + 1)(j + 2)
$$
as expected.

It remains to investigate the case $j = 1$ and $j = 2$.

For $j = 1$, by Theorem \ref{hyper-thm} (4), the inverse image of $0 \in E$ under $f_1: V_n \to E$ consists of a single point $0^n$. By Theorem \ref{hyper-thm} (5), the limiting system of equations \eqref{lim-eq} of tangent directions, has a unique non-degenerate solution in $\Bbb P^{n - 1}$ up to permutations. From this, we conclude that there are precisely $n!$ branches of $V_n \to E$ near $0^n$. For a point $b \in E^\times$ close to $0$, each branch will contribute a point $a$ with $a_1 = b$. In particular, $f_1$ is a finite morphism and $\deg f_1^* L = \deg f_1 = n!$.

For $j = 2$, we consider the inverse image of $0 \in E$ under $f_2: V_n \to E$. Namely $V_n \ni a \mapsto a_1 + a_2 = 0$. 

The point $a = 0$ again contributes degree $n!$ by a similar branch argument: Indeed, over each branch near $0^n$ we may represent $a = (a_i(t))$ by an analytic curve in $t$. Then condition $t_i + t_j \ne 0$ in Theorem \ref{hyper-thm} (5) implies that $t \mapsto a_1(t) + a_2(t) \in E$ is still locally biholomorphic for $t$ close to $0$. As a byproduct, since every irreducible component contains a branch near $0^n$, $f_2$ is necessarily a finite morphism and $\deg f_2^* L = \deg f_2$.

For those points $a \ne 0$ with $f_2(a) = 0$, we have $a_1 = -a_2$ and thus $a = -a$ by Theorem \ref{hyper-thm} (4). By Theorem \ref{t:nodal} we use $C$ as the coordinate and parameterize a (smooth) branch of $V_n$ near $a$ by $b(C) = (b_i(C))_{i = 1}^n$ with $b(0) = a$. In the proof of Theorem \ref{t:nodal} we see that $b_1'(0) = b_2'(0) \not\in \{0, \infty\}$ and $b_1'(0) + b_2'(0) \ne 0, \infty$, hence $f_2$ is unramified at $a$. The degree contribution at $a$ can thus be computed from counting points.

If $n = 2k$, by Proposition \ref{4lame} (1) the degree contribution from type O points $a = \{\pm a_1, \cdots, \pm a_k\}$ is given by
$$
(k + 1) \times (k  \times 2 \times  (n - 2)!),
$$
while the degree from the type II points $\{\pm a_1, \cdots, \pm a_{k - 1}, \tfrac{1}{2}\omega_i, \tfrac{1}{2} \omega_j \}$ is
$$
3\times k \times ((k - 1) \times 2 \times (n - 2)!).
$$
The sum is $2(4k^2 - 2k)(n - 2)! = 2n!$.

If $n = 2k + 1$, by Proposition \ref{4lame} (2), the degree contribution from type III points $\{\pm a_1, \cdots, \pm a_{k - 1}, \tfrac{1}{2} \omega_1, \tfrac{1}{2}\omega_2, \tfrac{1}{2} \omega_3\}$ is
$$
k \times ((k - 1) \times 2 \times (n - 2)!),
$$
while the type I points $\{\pm a_1, \cdots, \pm a_k, \tfrac{1}{2} \omega_i\}$ contribute
$$
3 \times (k + 1) \times (k \times 2 \times (n - 2)!).
$$
The sum is again $2(4k^2 + 2k)(n - 2)! = 2n!$.

The counting is valid even if $\bar X_n$ has nodal singularities. Thus in both cases we get the total degree $n! + 2n! = 3n!$ as expected. 
\end{proof}

To end this section, we notice that in Theorem \ref{hyper-thm} (5) we have $\sum_{i = 1}^n t_i \ne 0$ by the non-vanishing of Vandermonde determinant, hence we get 

\begin{proposition} \label{unram}
The map $\sigma_n$ is unramified at the infinity point $0^n \in \bar X_n$.
\end{proposition}

\section{The primitive generator $\z_n$} \label{primitive}
\setcounter{equation}{0}

\begin{definition} [Fundamental rational function] \label{d:zn}
Consider the function on $E^n$:
$$
\z_n(a_1, \ldots, a_n) := \zeta \Big( \sum_{i = 1}^n a_i \Big) - \sum_{i = 1}^n \zeta(a_i).
$$
$\z_n$ is a rational function on $E^n$ since it is meromorphic and periodic in each $a_i$.
\end{definition}

The importance of $\z_n$ is readily seen from investigation on the Green function equation \eqref{e:G}: Let $a_i = r_i \omega_1 + s_i \omega_2$. Then
\begin{equation} \label{e:znZ}
\begin{split}
-4\pi \sum \nabla G(a_i) &= \sum Z(a_i) = \sum (\zeta(r_i \omega_1 + s_i \omega_2) - r_i \eta_1 - s_i \eta_2) \\
&= \zeta(\sum a_i) - (\sum r_i) \eta_1 - (\sum s_i) \eta_2 - \z_n(a) \\
&= Z(\sum a_i) - \z_n(a).
\end{split}
\end{equation}
Hence $\sum_{i = 1}^n \nabla G(a_i) = 0 \Longleftrightarrow\z_n(a) = Z(\sigma_n(a))$. This links $\sigma_n(a)$ with $\z_n$. 

When no confusion should arise, we denote the restriction $\z_n|_{\bar X_n}$ also by $\z_n$. Then $\z_n$ is a rational function on $\bar X_n$ with poles along the fiber $\sigma_n^{-1}(0)$. Since $\z_1 \equiv 0$, we assume that $n \ge 2$ to avoid trivial situation.  

\begin{theorem} \label{main-thm}
There is a (weighted homogeneous) polynomial 
$$
W_n(\bz) \in \Bbb Q[g_2, g_3, \wp(\sigma), \wp'(\sigma)][\bz]
$$ 
of $\bz$-degree $\tfrac{1}{2} n (n + 1)$ such that for $\sigma = \sigma_n(a) = \sum a_i$, we have
$$
W_n(\z_n)(a) = 0.
$$
Indeed, $\z_n(a)$ is a primitive generator of the finite extension of rational function fields $K(\bar X_n)$ over $K(E)$ with $W_n(\bz)$ being its minimal polynomial. \footnote{The coefficients lie in $\Bbb Q$, instead of just in $\Bbb C$, follows from standard elimination theory and two facts (i) The equations of $\bar X_n$ is defined over $\Bbb Q[g_2, g_3]$ (cf.~\eqref{e:P}), and (ii) the addition map $E^n \to E$ is defined over $\Bbb Q$. In \S \ref{resultant}, we carry out the elimination procedure using resultant for another explicit presentation $\pi_n$ of $\sigma_n$.}
\end{theorem}

\begin{remark}
Since $\z_n$ has no poles over $E^\times$, it is indeed integral over the affine Weierstrass model of $E^\times$ with coordinate ring
$$
R(E^\times) = \Bbb C[x_0, y_0]/(y_0^2 - 4x_0^3 - g_2 x_0 - g_3),
$$ 
where $x_0 = \wp(\sigma)$ and $y_0 = \wp'(\sigma)$. Thus the major statement in Theorem \ref{t:primitive} is the claim that $\z_n$ is a primitive generator. 
\end{remark}

\begin{proof}
Since $\z_n \in K(\bar X_n)$, which is algebraic over $K(E)$ with degree $\tfrac{1}{2} n(n + 1)$ by Theorem \ref{degree}, its minimal polynomial $W_n(\bz) \in K(E)[\bz]$ exists with $d:= \deg W_n$ begin a factor of $\tfrac{1}{2} n(n + 1)$. 

Notice that for $\sigma_0 \in E$ being outside the branch loci of $\sigma_n: \bar X_n \to E$, there are precisely $\tfrac{1}{2} n(n + 1)$ different points $a = \{a_1, \cdots, a_n\} \in \bar X_n$ with $\sigma_n(a) = \sum a_i = \sigma_0$. Thus for the rational function $\z_n = \zeta(\sum a_i) - \sum \zeta(a_i) \in K(\bar X_n)$ to be a primitive generator, it is sufficient to show that $\z_n$ has exactly $\tfrac{1}{2}n(n + 1)$ branches over $K(E)$. That is, $\sum\zeta(a_i)$ gives different values for different choices of those $a$ above $\sigma_0$. Indeed, for any given $\sigma = \sigma_0$, the polynomial $W_n(\bz) = 0$ has at most $d$ roots. But now $\z_n(a)$ with $\sigma_n(a) = \sigma_0$ gives $\tfrac{1}{2} n(n + 1)$ distinct roots of $W_n(\bz)$, hence we must conclude $d = \tfrac{1}{2} n(n + 1)$ and $\z_n$ is a primitive generator.

Hence it is sufficient to show the following more precise result:
\begin{theorem} \label{key-claim}
Let $a, b \in Y_n$ and $(a_1, \cdots, a_n)$, $(b_1, \cdots, b_n) \in \Bbb C^n$ be representatives of $a$, $b$ such that
\begin{equation} \label{match}
\sum_{i = 1}^n a_i = \sum_{i = 1}^n b_i, \qquad \sum_{i = 1}^n \zeta(a_i) = \sum_{i = 1}^n \zeta(b_i).
\end{equation}

Suppose that $\sum \wp(a_i) \ne \sum \wp(b_i)$. Then $a, b$ are branch points of $Y_n \to \Bbb P^1$ corresponding to Lam\'e functions of the same type.
\end{theorem}

We emphasize that $\bar X_n$ is not required to be smooth. 

Theorem \ref{main-thm} follows immediately by choosing $\sigma_0$ outside the branch loci of $\bar X_n \to E$ and $\sigma_0 \not\in E[2]$. Indeed, let $a, b \in Y_n$ with $\sigma_n(a) = \sigma_n(b) = \sigma_0$ and $\z_n(a) = \z_n(b)$, or more precisely with conditions in \eqref{match} satisfied. By Theorem \ref{key-claim} we are left with the case $\sum \wp(a_i) = \sum \wp(b_i)$ but $a \ne b$. Then $a = -b$ by Theorem \ref{hyper-thm} (1), and in particular $\sigma_n(a) = -\sigma_n(b)$. Together with $\sigma_n(a) = \sigma_n(b)$ we conclude that $\sigma_0 = \sigma_n(a) = \sigma_n(b) \in E[2]$. This contradicts to the assumption $\sigma_0 \not \in E[2]$. Hence we must have $a = b$. 
\end{proof}

We will give two proofs of Theorem \ref{key-claim}. The first proof is longer but contains more information. 

Recall that the Hermite--Halphen ansatz in \eqref{ansatz}
$$
w_{\pm a}(z) = e^{\pm z\sum \zeta(a_i)} \prod_{i = 1}^n \frac{\sigma(z \mp a_i)}{\sigma(z)}
$$ 
are solutions to $w'' = (n(n + 1) \wp(z) + B_a)w =: I_1 w$, and 
$$
w_{\pm b}(z) = e^{\pm z\sum \zeta(b_i)} \prod_{i = 1}^n \frac{\sigma(z \mp b_i)}{\sigma(z)}
$$ 
are solutions to $w'' = (n(n + 1) \wp(z) + B_b)w =: I_2 w$. Then $q_{a, -b} := w_a w_{-b}$ and $q_{-a, b} := w_{-a} w_b$ are solutions to the fourth order ODE formed by the tensor product of the two Lam\'e equations. By assumption.
\begin{equation} \label{q-rat}
q_{a, -b}(z) = \prod_{i = 1}^n \frac{\sigma(z - a_i) \sigma(z + b_i)}{\sigma^2(z)}
\end{equation}
is an elliptic function since $\sum a_i = \sum b_i$. Similarly $q_{-a, b}(z) = q_{a, -b}(-z)$ is elliptic. In particular there exists an even elliptic function solution 
$$
Q := \tfrac{1}{2}(q_{a, -b} + q_{-a, b}) = (-1)^n \frac{\prod_{i = 1}^n \sigma(a_i) \sigma(b_i)}{z^{2n}} + \mbox{higher order terms}.
$$ 

\begin{lemma} \label{tensor-lame}
The fourth order ODE is given by
\begin{equation} \label{4th-ode}
q'''' -2(I_1 + I_2) q'' -6I'q' + ((B_a - B_b)^2 - 2I'')q = 0.
\end{equation}
Here $I = n(n + 1) \wp(z)$, $I_1 = I + B_a$ and $I_2 = I + B_b$.
\end{lemma}

\begin{proof}
This follows from a straightforward computation. Indeed,
\begin{equation*}
\begin{split}
q' &= w_1' w_2 + w_1 w_2', \\
q'' &= (I_1 + I_2) q + 2 w_1' w_2', \\
q''' &= 2I' q + (I_1 + I_2) q' + 2(I_1 w_1 w_2' + I_2 w_1' w_2).
\end{split}
\end{equation*}
Notice that if $a = b$ (or just $B_a = B_b$) then $I_1 = I_2$ and we stop here to get the third order ODE as the symmetric product of the Lam\'e equation. 

In general, we take one more differentiation to get 
\begin{equation*}
\begin{split}
q'''' &= 2I'' q + 4 I' q' + (I_1 + I_2) q'' + 2 I' q' + 2(I_1 + I_2) w_1' w_2' + 4 I_1 I_2 q \\
&= 2(I_1 + I_2) q'' + 6I' q' + (2I'' - (I_1 - I_2)^2) q.
\end{split}
\end{equation*}
This proves the lemma.
\end{proof}

Now we investigate the equation in variable $x = \wp(z)$. To avoid confusion, we denote $\dot{f} = \p f/\p x$ and $f' = \p f /\p z$. 

Let $y^2 = p(x) = 4x^3 - g_2 x - g_3$. Then $\wp' = y$, $\wp'' = 6\wp^2 - \tfrac{1}{2} g_2 = \tfrac{1}{2} \dot{p}(x)$. $\wp''' = 12 \wp \wp' = 12 xy$, $\wp'''' = 12 \wp'^2 + 12 \wp \wp'' = 12 p(x) + 6x\dot{p}(x)$. Also
\begin{equation*}
\begin{split}
q' &= \dot{q} \wp' = y\dot{q}, \\
q'' &= \ddot{q} \wp'^2 + \dot{q} \wp'' = p(x) \ddot{q} + \tfrac{1}{2} \dot{p}(x) \dot{q}, \\
q''' &= \dddot{q} \wp'^3 + 3 \ddot{q} \wp' \wp'' + \dot{q} \wp''', \\
q'''' &= \ddddot{q} \wp'^4 + 6 \dddot{q} \wp'^2 \wp'' + 3 \ddot{q} (\wp'')^2 + 4 \ddot{q} \wp' \wp''' + \dot{q} \wp'''' \\
&= p(x)^2 \ddddot{q} + 3p(x)\dot{p}(x) \dddot{q} + \big(\tfrac{3}{4} \dot{p}(x)^2 + 48 xp(x)\big) \ddot{q} + \big(12p(x) + 6 x \dot{p}(x) \big) \dot{q}.
\end{split}
\end{equation*}
By substituting these into (\ref{4th-ode}) and get the ODE in $x$:
\begin{equation} \label{4th-ODE}
\begin{split}
L_4\,q &:= p^2 \ddddot{q} + 3p \dot{p} \dddot{q} + \big( \tfrac{3}{4} \dot{p}^2 - 2(2(n^2 + n - 12) x + \beta) p\big) \ddot{q} \\
& \quad-\big((2(n^2 + n - 3)x + \beta)\dot{p} + 6(n^2 + n - 2) p\big) \dot{q} \\
&\qquad +\big(\alpha^2 - n(n + 1)\dot{p}\big) q = 0.
\end{split}
\end{equation}
where 
\begin{equation} \label{AB}
\alpha := B_a - B_b \quad \mbox{and} \quad\beta := B_a + B_b.
\end{equation}

For the rest of the proof, we want to discuss when $L_4\, q = 0$ with $\alpha \ne 0$ has a polynomial solution. Here $g_2$ and $g_3$ could be arbitrary, not necessarily satisfy the non-degenerate condition $g_2^3 - 27 g_3^2 \ne 0$. 

Suppose that $q(x)$ is a polynomial in $x$ of degree $m \ge 1$:
\begin{equation} \label{q-poly}
\begin{split}
q(x) = x^m - s_1 x^{m - 1} + s_2 x^{m - 2} - \cdots + (-1)^m s_m,
\end{split}
\end{equation}
which satisfies
\begin{equation} \label{L-mod}
\deg_x L_4\,q(x) \le 1.
\end{equation}
Then we can solve $s_j$ recursively in terms of $\alpha^2$, $\beta$ and $g_2$, $g_3$. 

Indeed, the top degree $x^{m + 2}$ in (\ref{4th-ODE}) has coefficient
\begin{equation*}
\begin{split}
&16 m(m - 1)(m - 2)(m - 3) + 144 m(m - 1)(m - 2) + 108m(m - 1) \\
&\quad - 16(n^2 + n - 12)m(m - 1) - 24(n^2 + n - 3)m \\
&\qquad - 24(n^2 + n - 2)m - 12n(n + 1) \\
&= (m - n) \Big( 4m^3 + (4n + 68) m^2 + (8n - 101) m + 3(n + 1)\Big),
\end{split}
\end{equation*}
which vanishes precisely when $m = n$. This we may assume that $m = n$.

The next order term $x^{n + 1}$ without the $s_1$ factor has coefficient
$$
-8n(n - 1)\beta - 12n\beta = -4n(2n + 1)\beta,
$$
and the coefficient of $-s_1 x^{n + 1}$ is given by
\begin{equation*}
\begin{split}
&16 (n - 1)(n - 2)(n - 3)(n - 4) + 144 (n - 1)(n - 2)(n - 3) \\
&\quad + 108(n - 1)(n - 2) - 16(n^2 + n - 12)(n - 1)(n - 2) \\
&\quad - 24(n^2 + n - 3)(n - 1) - 24(n^2 + n - 2)(n - 1) - 12n(n + 1) \\
&= -8n(2n - 1)(2n + 1).
\end{split}
\end{equation*}
Hence
\begin{equation} \label{s_1}
s_1 = \frac{\beta}{2(2n - 1)}.
\end{equation}

Inductively the $x^{n + 2 - i}$ coefficient in (\ref{4th-ODE}) gives recursive relations to solve $s_i$ in terms of $\beta$, $\alpha^2$ and $g_2, g_3$ for $i = 1, \ldots, n$. It implies that

\begin{lemma} \label{degrees}
For $i = 1, \ldots, n$, there is a polynomial expression
$$
s_i = s_i(\alpha^2, \beta, g_2, g_3) = C_i \beta^i + \cdots
$$ 
which is homogeneous of degree $i$ with $\deg \alpha =\deg \beta = 1$ and $\deg g_2 = 2$, $\deg g_3 = 3$. Moreover, $C_i$ is a non-zero rational number.
\end{lemma}

A much detailed description will be given in the proof of Lemma \ref{G1G2} and the precise value of $C_i$ can be determined (from (\ref{recur})).

There are still two remaining terms in (\ref{L-mod}), that is,
\begin{equation} \label{consistency}
L_4\,q = F_1(\alpha, \beta, g_2, g_3) x + F_0(\alpha, \beta, g_2, g_3).
\end{equation}


The basic structure of the consistency equations is described by the following two lemmas:

\begin{lemma} \label{F0F1}
We have
\begin{equation*}
\begin{split}
F_1(\alpha, \beta) &= \alpha^{2} G_1(\alpha, \beta) = \alpha^2 ((-1)^{n - 1} s_{n - 1}(\alpha^2, \beta, g_2, g_3) + \cdots), \\
\quad F_0(\alpha, \beta) &= \alpha^{2} G_0(\alpha, \beta) = \alpha^2 ((-1)^n s_n(\alpha^2, \beta, g_2, g_3) + \cdots).
\end{split}
\end{equation*}
The remaining terms have either $g_2$ or $g_3$ as a factor, hence with lower $\alpha, \beta$ degree.
\end{lemma}

\begin{proof}
Equation (\ref{consistency}) gives 
\begin{equation*}\begin{split}
F_1(\alpha, \beta) &= (-1)^{n - 1} \alpha^2 s_{n - 1} + \mbox{terms in $s_1, \cdots, s_{n - 2}$}, \\
F_0(\alpha, \beta) &= (-1)^n \alpha^2 s_n + \mbox{terms in $s_1, \cdots, s_{n - 1}$}. 
\end{split}
\end{equation*}
We note that if $\alpha = 0$, then for any $\beta$ there is a solution $q(x)$ to $L_4(q) = 0$ which is a polynomial in $x$ of degree $n$. 

Indeed $q(x) = \prod_{i = 1}^n (x - x_i)$, with $\beta = 2(2n - 1) \sum_{i = 1}^n x_i$, which comes from the Lam\'e equation (see \cite{CLW, Whittaker}). Thus $F_1(0, \beta) = 0 = F_0(0, \beta)$. Since $F_i$ depends on $\alpha^2$, we have $F_i(\alpha, \beta) = \alpha^2 G_i(\alpha, \beta)$, $i = 0, 1$, for some homogeneous polynomials $G_0$, $G_1$ in $\alpha^2$, $\beta$, $g_2$, $g_3$ of degree $n$ and $n - 1$ respectively, and $G_i$'s can be written as
\begin{equation*}
\begin{split}
G_1(\alpha, \beta) &= (-1)^{n - 1} s_{n - 1} + \cdots, \\
G_0(\alpha, \beta) &= (-1)^n s_n + \cdots.
\end{split}
\end{equation*}

To see the dependence of the remaining terms on $g_2$ and $g_3$, we let $g_2 = 0 = g_3$, and then $L_4(q) \equiv \alpha^2 ((-1)^{n - 1} s_{n - 1} x + (-1)^n s_n) \pmod{x^2}$ because both $p(x) = 4x^3$ and $\dot{p}(x) = 12 x^2$ vanish modulo $x^2$. Thus we have $F_1(\alpha, \beta) = (-1)^{n - 1}\alpha^2 s_{n - 1}$ and $F_0(\alpha, \beta) = (-1)^n \alpha^2 s_n$ whenever $g_2 = 0 = g_3$. This proves the lemma.
\end{proof}


\begin{lemma} \label{G1G2}
The polynomials $G_1$ and $G_0$ have no common factors for any $g_2, g_3$.
\end{lemma}

\begin{proof}
We consider first the special case $g_2 = g_3 = 0$. Then (\ref{L-mod}) becomes
\begin{equation} \label{4th-ODE-0}
\begin{split}
&16x^6 \ddddot{q} + 144x^5 \dddot{q} + \big( 108x^4 - 8x^3(2(n^2 + n - 12) x + \beta) \big) \ddot{q} \\
& \quad-\big(12x^2(2(n^2 + n - 3)x + \beta) + 24x^3(n^2 + n - 2) \big) \dot{q} \\
&\qquad +\big(\alpha^2 - 12n(n + 1)x^2\big) q \equiv 0 \pmod{\Bbb C \oplus \Bbb C x}.
\end{split}
\end{equation}

The coefficient of $x^{n - k}$, $k = 0, \ldots, n - 2$, gives recursive equation
\begin{equation} \label{recur}
(-1)^k (m_k\, s_{k + 2} + n_k \beta\, s_{k + 1} + \alpha^2 s_k) = 0,
\end{equation}
where the constants $m_k$ and $n_k$ are given by 
\begin{equation*}
\begin{split}
m_k &= 16 (n - (k + 2))(n - (k  + 3))(n - (k + 4))(n - (k + 5)) \\
&\quad+ 144 (n - (k + 2))(n - (k  + 3))(n - (k + 4)) \\
&\qquad+(108 - 16(n^2 + n - 12)) (n - (k + 2))(n - (k  + 3)) \\
&\qquad\quad- 24(2n^2 + 2n - 5) (n - (k + 2)) - 12 n(n + 1) \\
&= -4(k + 2) (2n - (k + 1)) (2n - (2k + 1)) (2n - (2k + 3)), \\
n_k &= (8 (n - (k + 1))(n - (k + 2)) + 12(n - (k + 1))) \\
&= 4(n - (k - 1))(n - (k + 1)).
\end{split}
\end{equation*}
Since $k \le n - 2$, we have $m_k \ne 0$ and $n_k \ne 0$.

Let $\gamma(\alpha, \beta)$ be a non-trivial common factor of both $G_1$ and $G_0$. 

In the case $g_2 = g_3 = 0$ we have $G_1 = (-1)^{n - 1}s_{n - 1}$ and $G_0 = (-1)^n s_n$. Then $\gamma$ and $\alpha$ are co-prime, because if $\alpha = 0$ then $s_{n - 1}(0, \beta) = c_{n - 1} \beta^{n - 1}$ and $s_n(0, \beta) = c_n \beta^n$ for some non-zero constants $c_{n - 1}$ and $c_n$. By (\ref{recur}) for $k = n - 2$, we have $\gamma \mid s_{n - 2}(\alpha^2, \beta, 0, 0)$ too. By induction on $k$ for $k = n - 3, \ldots, 0$ in decreasing order we conclude that $\gamma \mid s_0 = 1$, which leads to a contradiction.

For $g_2, g_3 \in \Bbb C$, we see by Lemma \ref{F0F1} that the leading terms of $G_1$, $G_0$, as polynomials of $\alpha$ and $\beta$, are $(-1)^{n - 1} s_{n - 1}(\alpha^2, \beta, 0, 0)$ and $(-1)^{n} s_{n}(\alpha^2, \beta, 0, 0)$ respectively. Since $s_{n - 1}(\alpha^2, \beta, 0, 0)$ and $s_n(\alpha^2, \beta, 0, 0)$ are co-prime, we conclude that $G_1(\alpha, \beta, g_2, g_3)$ and $G_0(\alpha, \beta, g_2, g_3)$ are also co-prime. The proof is complete.
\end{proof}

\begin{proposition} \label{lame-pair}
The common zeros of $G_1 = 0$ and $G_0 = 0$ are precisely given by the pair of branch points $(a, b)$ corresponding to Lame functions of the same type. If $\bar X_n$ is non-singular, there are exactly $n(n - 1)$ such ordered pairs $(a, b)$'s.
\end{proposition}

\begin{proof}
It suffices to prove the (generic) case that $\bar X_n$ is non-singular, namely the case that all the Lam\'e functions are distinct. The general case follows from the non-singular case by a limiting argument. 

For any two Lam\'e functions $w_a$, $w_b$ of the same type, it is easy to see that we may arrange the representatives of $a$ and $b$ so that \eqref{match} holds. It follows that $q := q_{a, -b} = q_{-a, b}$ (see \eqref{q-rat}) is an even elliptic function solution to \eqref{4th-ode}, or equivalently $q(x)$ is a polynomial solution to $L_4\, q(x) = 0$.

From the above discussion, $(\alpha, \beta)$ must be a common root of $G_1$ and $G_0$ (where $\alpha = B_a - B_b$, $\beta = B_a + B_b$). By Lemma \ref{degrees} and \ref{F0F1}, we have $\deg G_1 = n - 1$ and $\deg G_0 = n$ and $G_1$, $G_0$ are co-prime to each other by Lemma \ref{G1G2}. Hence by Bezout theorem there are at most $n(n - 1)$ common roots.

On the other hand, the number of such ordered pairs can be determined by Proposition \ref{4lame}. Indeed, if $n = 2k$ is even, then we have 
$$
(k + 1)k + 3k(k - 1) = 4k^2 - 2k = n(n - 1)
$$
such pairs. If $n = 2k + 1$ is odd, the number of pairs is given by
$$
k(k - 1) + 3(k + 1)k = 4k^2 + 2k = n(n - 1).
$$
 Hence in all cases the number of ordered pairs coming from the Lam\'e functions of the same type agrees with the Bezout degree of the polynomial system defined by $G_1 = 0 = G_0$. Thus these $n(n - 1)$ pairs form the zero locus as expected (and there is no infinity contribution). 
\end{proof}

The above discussions from Lemma \ref{tensor-lame} to Proposition \ref{lame-pair} constitute a complete proof of Theorem \ref{key-claim}. Here is a summary: We already know that $Q$ is an even elliptic function with singularity only at $0 \in E$. Thus
$$
Q(x) = c \prod_{i = 1}^n (\wp(z) - \wp(c_i)) =: c \prod_{i = 1}^n (x - x_i) 
$$
is a polynomial solution to the ODE (\ref{4th-ODE}) with $\alpha = B_a - B_b$, $\beta = B_a + B_b$.

Since $\alpha = B_a - B_b \ne 0$, by Lemma \ref{F0F1} $(\alpha, \beta)$ must be a common root of $G_1(\alpha, \beta) = 0 = G_0(\alpha, \beta)$. Then Proposition \ref{lame-pair} says that $(\alpha, \beta)$ is pair of Lam\'e functions of the same type. This proves Theorem \ref{key-claim}. \medskip

For future reference, we combine Theorem \ref{key-claim} and Proposition \ref{lame-pair} into the following statement on a fourth order ODE which arises from the \emph{tensor product of two different (integral) Lam\'e equations} with the same parameter $n$. 

Due to its importance, we will give a second (shorter and more direct) proof of the part corresponding to Theorem \ref{key-claim}. 

\begin{theorem}
Let $I(z) = n(n + 1) \wp(z)$. The fourth order ODE
\begin{equation} \label{4th-ode'}
q''''(z) - 2(I + \beta) q''(z) - 6I' q'(z) + (\alpha^2 - 2I'') q(z) = 0
\end{equation}
with $\alpha \ne 0$ has an elliptic function solution if and only if $(\alpha, \beta)$ is a pair of common root to $G_0(\alpha, \beta) = 0$ and $G_1(\alpha, \beta) = 0$. Moreover, this solution must be even.
\end{theorem} 

\begin{proof}[Second Proof to Theorem \ref{key-claim}]
Following the definition of $q_{a, -b}(z)$ in \eqref{q-rat}, we now consider the odd elliptic solution to \eqref{4th-ode'} ($=$ (\ref{4th-ode})) instead: 
$$
q(z) = \tfrac{1}{2}(q_{a, -b}(z) - q_{-a, b}(z)),
$$
which has a pole of order $3 + 2l$ at $0 \in E$ with $l \le n - 2$. Thus $q(z)/\wp'(z)$ is an even elliptic function with the only pole at $0$ since $q(\tfrac{1}{2}\omega_i) = 0$ for $1 \le i \le 3$. If $q(z)$ does not vanish completely, then
$$
q(z) = c \wp'(z) \prod_{i = 1}^l (\wp(z) - \wp(c_i)) =: c \wp'(z) f(\wp(z)),  
$$ 
where $f(x) = \prod_{i = 1}^l (x - \wp(c_i)) = x^l - s_1 x^{l - 1} + \cdots + (-1)^l s_l$.

By Lemma \ref{tensor-lame}, $q(z)$ satisfies
\begin{equation} \label{4th-ode1}
\begin{split}
&q''''(z) - 2(\beta + 2n(n + 1) \wp(z)) q''(z) \\
&\qquad - 6n(n + 1)\wp'(z) q'(z) + (\alpha^2 - 2n(n + 1) \wp''(z)) q(z) = 0.
\end{split}
\end{equation}
By straightforward calculations, we can compute all derivatives of $q$ in terms of derivatives of $\wp(z)$ and $f'(x)$. For example,
\begin{equation*}
\begin{split}
q'(z) &= \wp''(z) f(x) + \wp'(z)^2 f'(x),\\
q''(z) &= \wp'''(z) f(x) + 3\wp''(z) \wp'(z) f'(x) + \wp'(z)^3 f''(x), \quad\mbox{etc.}
\end{split}
\end{equation*}
Then (\ref{4th-ode1}) is equivalent to 
\begin{equation*}
\begin{split}
&f(x)\Big((360 - 96n(n + 1))x^2 - 24 \beta x + (4n(n + 1) - 18)g_2 + \alpha^2\Big) \\
+& f'(x) \Big((1320 - 96n(n + 1))x^3 - 36 \beta x^2 \\
&\qquad \qquad + (12n(n + 1) - 150)g_2 x + (6n(n + 1)- 60) g_3 + 3\beta g_2\Big) \\
+& f''(x)\Big( (1020 - 16n(n + 1))x^4 - 8\beta x^3 + (4n(n + 1) - 210)g_2 x^2 \\ 
&\qquad\qquad + (2\beta g_2 + (4n(n + 1) - 120) g_3) x  + 2\beta g_3 + \tfrac{15}{4} g_2^2\Big) \\
+& f'''(x) (60 x^2 - 30 g_2) (4x^3 - g_2 x - g_3) \\
+& f''''(x) (4x^3 - g_2 x - g_3)^2 = 0.
\end{split}
\end{equation*}

By comparing the coefficients of $x^{l + 2}$, we obtain
\begin{equation*}
\begin{split}
&(360 - 96n(n + 1)) + l(1320 - 96 n(n + 1)) + l(l - 1) (1020 - 16 n(n + 1)) \\
&\quad + 240l(l - 1)(l - 2) + 16 l(l - 1)(l -2)(l - 3) = 0.
\end{split}
\end{equation*}
After simplification, this is reduced to
$$
4n(n + 1) = (2l + 3)(2l + 5),
$$
which obviously leads to a contradiction since the RHS is odd. Therefore we must have $q \equiv 0$ from the beginning. That is, $\{a_i, -b_i\} = \{-a_i, b_i\}$.

If one of $a, b$ does not correspond to a Lam\'e function, say $a \in X_n$, then $\{a_1, \cdots, a_n\} \cap \{-a_1, \cdots, -a_n\} = \emptyset$ and we conclude that $\{a_i\} = \{b_i\}$. Otherwise $a$ and $b$ correspond to Lam\'e functions of the same type. 
\end{proof}

%
%
%

\begin{example} \label{ex:n=2}
For $n = 2$, $\beta = B_a + B_b$, $\alpha = B_a - B_b$, we have 
\begin{equation*}
\begin{split}
s_1 = \tfrac{1}{6}\beta, \qquad s_2 = \tfrac{1}{36} \beta^2 + \tfrac{1}{72}\alpha^2 - \tfrac{1}{4} g_2.
\end{split}
\end{equation*}

The first compatibility equation from $x^1$ is
$$
s_1(\alpha^2 + 36 g_2) - 6\beta g_2 = 0.
$$
After substituting $s_1$ we get
\begin{equation} \label{c2-1}
\tfrac{1}{6}\alpha^2 \beta = 0.
\end{equation}

The second compatibility equation from $x^0$ is
$$
s_2(\alpha^2 + 6g_2) - s_1(\beta g_2 + 24 g_3) + 4 \beta g_3 + \tfrac{3}{2} g_2^2 = 0.
$$
By substituting $s_1$, $s_2$ and noticing the (expected) cancellations we get
\begin{equation} \label{c2-2}
\alpha^2 (\tfrac{1}{36} \beta^2 + \tfrac{1}{72} \alpha^2 - \tfrac{1}{6} g_2) = 0.
\end{equation}

If $B_a \ne B_b$ then (\ref{c2-1}) implies that $B_b = -B_a$ and then (\ref{c2-2}) leads to
$$
B_a^2 = 3g_2 \Longrightarrow \wp(a_1) + \wp(a_2) = \pm \sqrt{g_2/3}.
$$
By Example \ref{lame-curve} (2), such $a \in \bar X_2$ lies in the branch loci of the hyperelliptic (Lam\'e) curve. In particular, $a, b\in \sigma^{-1}(0)$ and they are excluded by the assumption in Proposition \ref{key-claim}. Denote by $\wp(\pm q_{\pm}) = \pm \sqrt{g_2/12}$. Then $a := \{q_+, -q_+\} \ne b := \{q_-, -q_-\}$ unless $g_2 = 0$. When $g_2 \ne 0$, $\z_2$ fails to distinguish the two points $a$ and $b$. When $g_2 = 0$ (equivalently $\tau = e^{\pi i/3}$), $a = b$ becomes a (singular) branch point for $\sigma: \bar X_2 \to E_\tau$.
\end{example}

\begin{example}
For $n = 3$, $\beta = B_a + B_b$, $\alpha = B_a - B_b$. Then
\begin{equation*}
\begin{split}
s_1 &= \tfrac{1}{10} \beta,\\
s_2 &= \tfrac{1}{600}(4\beta^2 + \alpha^2 - 150 g_2), \\
s_3 &= \tfrac{1}{3600} (2 \beta^3 + 3\alpha^2\beta - 120\beta g_2 + 900g_3).
\end{split}
\end{equation*}

The two compatibility equations from $x^1$ and $x^0$ are
\begin{equation*}
\begin{split}
0 &=\tfrac{1}{600}\alpha^2(4\beta^2 + \alpha^2 + 60g_2),\\
0 &= \tfrac{1}{3600} \alpha^2(2\beta^3 + 3\alpha^2 \beta - 90 \beta g_2 + 540 g_3).
\end{split}
\end{equation*}
If $\alpha \ne 0$ then $\alpha^2 = -4\beta^2 - 60g_2$ and the second equation becomes
$$
\beta^3 + 27 g_2 \beta - 54 g_3 = 0.
$$
It is clear that there are only finite solutions $(B_a, B_b)$'s to this, though it may not be so straightforward to see that these 6 solution pairs (for generic tori) come from the branch loci as proved in Proposition \ref{lame-pair}.
\end{example}

\section{Pre-modular forms $Z_n(\sigma; \tau)$} \label{new-Zn}

We call a real analytic function in $(\sigma, \tau) \in \Bbb C \times \Bbb H$ \emph{pre-modular} if it is (holomorphic and) modular in $\tau$ for $\Gamma(N)$ whenever we fix $\sigma \pmod{\Lambda_\tau} \in E_\tau[N]$. Theorem \ref{main-thm} and Hecke's theorem on $Z$ \cite{Hecke} (cf.~\eqref{e:Hecke}) then imply 

\begin{corollary} \label{c:pre-mod}
$Z_n(\sigma; \tau) := W_n(Z)(\sigma; \tau)$ is pre-modular of weight $\tfrac{1}{2} n (n + 1)$, with $Z$, $\wp(\sigma)$, $\wp'(\sigma)$, $g_2$, $g_3$ being of weight 1, 2, 3, 4, 6 respectively.
\end{corollary}

Now we prove Theorem \ref{t:corr}.

We call the $2n + 1$ branch points $a \in Y_n \setminus X_n$ \emph{trivial critical points} since $a = -a$ and the Green equation \eqref{e:G} holds trivially. They satisfy a nice compatibility condition with the case $n = 1$ under the addition map:

\begin{lemma} \label{trivial-corr}
Let $a = \{a_1, \cdots, a_n\} \in Y_n$ be a solution to the Green equation $\sum_{i = 1}^n \nabla G(a_i) = 0$. Then $a$ is trivial, i.e.~$a = -a$, if and only if $\sigma_n(a) \in E[2]$.
\end{lemma}

\begin{proof}
If $a$ is trivial, then $\sigma_n(a) \in E[2]$ clearly. 
If $a$ is non-trivial, i.e.~$a \in X_n$, by \eqref{type-II}, it gives rise to a type II developing map $f$ with
\begin{equation*}
\begin{split}
f(z + \omega_1) = e^{-4\pi i\sum_i s_i} f(z), \qquad
f(z + \omega_2) = e^{4\pi i \sum_i r_i} f(z).
\end{split}
\end{equation*}
Here $a_i = r_i \omega_1 + s_i \omega_2$ for $i = 1, \ldots, n$.

If $\sigma_n(a) \in E[2]$, then both exponential factors reduce to one and we conclude that $f(z)$ is an elliptic function on $E$. Notice that the only zero of $f'(z)$ is at $z = 0$ which has order $2n$, and the only poles of $f'(z)$ are at $-a_i$ of order 2, $i = 1, \ldots, n$. This forces that $\sigma_n(a) \equiv 0 \pmod\Lambda$ and 
$$
f'(z) = \sum\nolimits_{j = 1}^n E_j \wp(z + a_j) + C_1
$$
for some constants $E_1, \ldots, E_n$ and $C_1$, since $f'$ is residue free. Then
$$
f(z) = -\sum\nolimits_{j = 1}^n E_j \zeta(z + a_i) + C_1 z + C_2
$$
for some constant $C_2$. But $f(z)$ is elliptic, which implies that $C_1 = 0$ and $\sum_{j = 1}^n E_j = 0$. Now $f^{2k - 1}(0) = 0$ for $k = 1, \ldots, n$ leads to a system of linear equations in $E_j$'s (c.f.~{\cite[Lemma 2.5]{CLW}}):
$$
\sum\nolimits_{j = 1}^n \wp^k(a_j) E_j = 0, \qquad k = 1, \ldots, n.
$$
But then $\wp(a_i) \ne \wp(a_j)$ for $i \ne j$ forces that $E_j = 0$ for all $j$. This is a contradiction and so we must have $\sigma_n(a) \not\in E[2]$.
\end{proof}

The following theorem completes the proof of Theorem \ref{t:corr}: 

\begin{theorem}[Extra critical points vs zeros of pre-modular forms] {\ } \label{zero-pre-mod}
\begin{itemize}
\item[(i)]
Given $\sigma_0 \in E_\tau \setminus E_\tau[2]$ with $Z_n(\sigma_0; \tau) = 0$, there is a unique $a \in X_n$ such that $\sigma_n(a) = \sigma_0$ and $\z_n(a) = Z(\sigma_0)$.
\item[(ii)]
Conversely, if $a \in X_n$ and $\z_n(a) = Z(\sigma(a))$, then $Z_n(\sigma(a); \tau) = 0$ and $\sigma_n(a) \not\in E_\tau[2]$.
\end{itemize}
\end{theorem}

\begin{proof}
(i) For any given $\sigma_0$, by substituting $\sigma$ by $\sigma_0$ in $W_n(\bz)$, we get a polynomial $W_{n, \sigma_0}(\bz)$ of degree $\tfrac{1}{2} n(n + 1)$. Since $W_n(\bz)$ is the minimal polynomial of the rational function $\z_n \in K(\bar X_n)$ over $K(E)$, those $\z_n(a)$ with $a \in \bar X_n$ and $\sigma_n(a) = \sigma_0$ give precisely all the roots of $W_{n, \sigma_0}(a)$, counted with multiplicities.

Now $Z(\sigma_0)$ is a root of $W_{n, \sigma_0}(\bz)$ with $\sigma_0 \not\in E[2]$, hence there is a point $a \in X_n$ corresponds to it, i.e.~$Z(\sigma_0) = \z_n(a)$ with $\sigma_n(a) = \sigma_0$, which is unique by Theorem \ref{key-claim}. Notice that if $a \in \bar X_n \setminus X_n$ then $a  = -a$ and then $\sigma_n(a) \in E[2]$. So in fact we must have $a \in X_n$.

(ii) It is clear that $Z_n(\sigma(a)) \equiv W_n(Z(\sigma(a)) = W_n(\z_n(a)) = 0$. Since $a \in X_n$, by \eqref{e:znZ} we have $\sum_{i = 1}^n \nabla G(a_i) = 0$. But since $a$ is non-trivial ($a \in X_n$ by assumption), Lemma \ref{trivial-corr} implies that $\sigma_n(a) \not\in E[2]$.  
\end{proof}

We present below an extended version of Theorem \ref{t:corr} in terms of \emph{monodromy groups of Lam\'e equations}. The original case of mean field equations corresponds to the case with \emph{unitary monodromy} (cf.~\cite{CLW}).

Let $a = \{a_1, \cdots, a_n\}\in X_n$, $B_a = (2n - 1) \sum_{i = 1}^n \wp(a_i)$ and $w_a$, $w_{-a}$ be the independent ansatz solutions \eqref{ansatz} to $w'' = (n(n + 1) \wp(z) + B_a) w$. From \eqref{e:sigma-trans}, one calculate easily that the monodromy matrices are given by
\begin{equation} \label{e:monodromy}
\begin{split}
\begin{pmatrix} w_a \\ w_{-a} \end{pmatrix}(z + \omega_1) &= \begin{pmatrix} e^{-2\pi i r} & 0 \\ 0 & e^{2\pi i r} \end{pmatrix} \begin{pmatrix} w_a \\ w_{-a} \end{pmatrix} (z), \\
\begin{pmatrix} w_a \\ w_{-a} \end{pmatrix}(z + \omega_2) &= \begin{pmatrix} e^{2\pi i s} & 0 \\ 0 & e^{-2\pi i s} \end{pmatrix} \begin{pmatrix} w_a \\ w_{-a} \end{pmatrix} (z),\end{split}
\end{equation} 
where the two \emph{complex numbers} $r, s \in \Bbb C$ are uniquely determined by
\begin{equation} \label{e:rs}
r \omega_1 + s \omega_2 = \sigma(a) = \sum_{i = 1}^n a_i, \qquad r \eta_1 + s \eta_2 = \sum_{i = 1}^n \zeta(a_i).
\end{equation}
The system is non-singular by the Legendre relation $\omega_1 \eta_2 - \omega_2 \eta_1 = -2\pi i$ .

The next lemma extends Lemma \ref{trivial-corr}:

\begin{lemma} \label{l:c-corr}
Let $a \in X_n$ with $(r, s)$ given by \eqref{e:rs}. Then $(r, s) \not\in \frac{1}{2} \Bbb Z^2$.
\end{lemma}

\begin{proof}
If $(r, s) \in \frac{1}{2} \Bbb Z^2$ then $f := w_a/w_{-a}$ is elliptic by \eqref{e:monodromy}. Since
$$
f' = \frac{w_a' w_{-a} - w_a w_{-a}'}{w_a^2} = \frac{C}{w_a^2},
$$
we find that $z = 0$ is the only zero of $f'(z)$, which has order $2n$. The proof of Lemma \ref{trivial-corr} for this $f$ goes through and leads to a contradiction. 
\end{proof}

Now we consider $Z_{r, s}(\tau)$ in \eqref{e:Hecke} but with $r, s, \in \Bbb C$, and define
\begin{equation} \label{e:cZ}
Z_{n;\, r, s}(\tau) := W_n(Z_{r, s})(r + s\tau; \tau), \qquad r, s \in \Bbb C.
\end{equation}
It reduces to $Z_n(\sigma; \tau)$ for $\sigma = r + s \tau$ when $r, s \in \Bbb R$ (see \cite{CKLW} for its role in the isomonodromy problems and Painleve VI equations).

By substituting $Z_n(\sigma; \tau)$ with $Z_{n;, r, s}(\tau)$ and using Lemma \ref{l:c-corr} in place of Lemma \ref{trivial-corr}, the proof of Theorem \ref{zero-pre-mod} also leads to:

\begin{theorem} \label{t:lame-mono}
Let $r, s \in \Bbb C$. Then any non-trivial solution $\tau$ to $Z_{n;\, r, s}(\tau) = 0$, i.e.~with $r + s \tau \pmod{\Lambda_\tau} \not \in E_\tau[2]$, corresponds to an $a = (a_1, \ldots, a_n)\in \Bbb C^n$ such that $a \pmod{\Lambda_\tau} \in X_n(\tau)$ and 
$$
\sum_{i = 1}^n a_i = r + s\tau, \qquad \sum_{i = 1}^n \zeta(a_i; \tau) = r \eta_1(\tau) + s \eta_2(\tau).
$$
Equivalently, by \eqref{e:rs}, the Lame equation $w'' = (n(n + 1) \wp(z; \Lambda_\tau) + B_a) w$ has its monodromy representation given by \eqref{e:monodromy}.
\end{theorem}

We leave the straightforward justifications to the interested reader.

\section{An explicit determination of $Z_n$} \label{resultant}
\setcounter{equation}{0}

From the equations of $\bar X_n \subset {\rm Sym}^n E$ (cf.~\eqref{e:P}) and the recursively defined algebraic formula of the addition map $E^n \to E$, in principle it is possible to compute $W_n$ and hence $Z_n$ by \emph{elimination theory} (cf.~\cite{Hassett}). However we shall present a more direct approach on this to reveal more structures inside it.

Besides the Hermite--Halphen ansatz \eqref{ansatz}, there is another ansatz, the \emph{Hermite--Krichever ansatz}, which can also be used to construct solutions to the integral Lam\'e equation \eqref{lame}. It takes the form
\begin{equation} \label{HK}
\psi(z) := \Big(U(\wp(z)) + V(\wp(z)) \frac{\wp'(z) + \wp'(a_0)}{\wp(z) - \wp(a_0)}\Big) \frac{\sigma(z - a_0)}{\sigma(z)} e^{(\zeta(a_0) + \kappa)z},
\end{equation}
where $U(x)$ and $V(x)$ are polynomials in $x$, $a_0 \in E^\times$, and $\kappa \in \Bbb C$ is a constant. As usual, we set $(x, y) = (\wp(z), \wp'(z))$ and $(x_0, y_0) = (\wp(a_0), \wp'(a_0))$ to be the corresponding algebraic coordinates.

Notice that (\ref{HK}) makes sense since $\psi$ only has poles at $z = 0$ (the one at $z = a_0$ from $(\wp(z) - \wp(a_0))^{-1}$ cancels with the zero from $\sigma(z - a_0)$). Moreover, in order for ${\rm ord}_{z = 0}\, \psi(z) = -n$, we must have 

\begin{lemma} [Degree constraints] {\ }
\begin{itemize}
\item[(i)] If $n = 2m$ with $m \in \Bbb N$ then $\deg U \le m - 1$ and $\deg V = m - 1$. 

\item[(ii)] If $n = 2m + 1$ with $m \in \Bbb N \cup \{0\}$ then $\deg U = m$ and $\deg V \le m - 1$.
\end{itemize}
\end{lemma}

By an obvious normalization, in case (i) we may assume that $U(x) = \sum_{i = 0}^{m - 1} u_i x^i$, $V(x) = \sum_{i = 0}^{m - 1} v_i x^i$ with $v_{m - 1} = 1$, and in case (ii) $U(x) = \sum_{i = 0}^{m} u_i x^i$ with $u_m = 1$ and $V(x) = \sum_{i = 0}^{m - 1} v_i x^i$. In both cases, the requirement that $\psi(z)$ satisfies \eqref{lame} leads to recursive relations on $u_i$'s and $v_i$'s. In doing so, it is more convenient to work on the algebraic coordinates. This had been carried out by Maier in \cite[\S 4]{Maier}. The following is a summary:

In case (i) the recursion determines $v_i$ ($v_{m - 1} = 1$) and then $u_i$ for $i = m - 1, m - 2, \cdots$  in decreasing order. In case (ii) it starts with $u_m = 1$ and determines $v_i$ and then $u_i$ for $i = m - 1, m -2, \cdots$. There are two compatibility equations coming from $u_{-1}(B, \kappa, x_0, y_0) = 0$ and $v_{-1}(B, \kappa, x_0, y_0) = 0$. The two parameters $x_0, y_0$ satisfy $y_0^2 = 4 x_0^3 - g_2 x_0 - g_3$. Hence there are four variables $(B, \kappa, x_0, y_0) \in \Bbb C^4$ which are subject to three polynomial equations. By taking in to account the limiting cases with $(x_0, y_0) = (\infty, \infty)$, this recovers the Lame curve $\bar Y_n$, which was denoted by $\Gamma_\ell$ in \cite{Maier} with $\ell = n$.

There are four natural coordinate projections (rational functions) $\bar Y_n \to \Bbb P^1$, namely $B, \kappa, x_0$ and $y_0$ respectively. The first one $B: \bar Y_n \to \Bbb P^1$ is simply the hyperelliptic structure map. The main result in \cite{Maier} is an explicit description of the other 3 maps in terms of the coordinates $(B, C)$ on $\bar Y_n$: 

\begin{theorem} [{\cite[Theorem 4.1]{Maier}}] \label{t:xyk}
For all $n \in \Bbb N$ and $i \in \{1, 2, 3\}$,
\begin{equation} \label{xyk}
\begin{split}
x_0(B) &= e_i + \frac{4}{n^2(n + 1)^2} \frac{l_i(B) lt_i(B)^2}{l_0(B) lt_0(B)^2}, \\
y_0(B, C) &= \frac{16}{n^3(n + 1)^3} \frac{C}{c_n} \frac{lt_1(B) lt_2(B) lt_3(B)}{l_0(B)^2 lt_0(B)^3}, \\
\kappa(B, C) &= -\frac{(n - 1)(n + 2)}{n(n + 1)} \frac{C}{c_n} \frac{l_\theta(B)}{l_0(B) lt_0(B)}.
\end{split}
\end{equation}
The formula for $x_0(B)$ is independent of the choices of $i$.

All the factors lie in $\Bbb Q[e_1, e_2, e_3, g_2, g_3, B]$ and are \emph{monic} in $B$. They are homogeneous with weights of $B, e_i, g_2, g_3$ being $1, 1, 2, 3$ respectively.  \end{theorem}

As a simple consistency check, we have $C^2 = \ell_n(B)$ by Proposition \ref{4lame}. 

In \eqref{xyk}, $lt_j(B)$, $j = 0, 1, 2, 3$, are the \emph{twisted Lam\'e polynomials} whose zeros correspond to solutions to \eqref{lame} given by the Hermite--Krichever ansatz with $\kappa \ne 0$ and $a_0 = 0$, $\tfrac{1}{2} \omega_1$, $\tfrac{1}{2} \omega_2$, $\tfrac{1}{2} \omega_3$ respectively, i.e.~$(x_0, y_0) = (\infty, \infty)$, $(e_1, 0)$, $(e_2, 0)$, $(e_3, 0)$ respectively. 

The polynomial $l_\theta(B)$ is the \emph{theta-twisted polynomial} whose roots correspond to the case $\kappa = 0$ and $a_0 \not \in E[2]$. (For $\kappa = 0$ and $a_0 \in E[2]$ they correspond to the \emph{ordinary} Lam\'e polynomials $l_i(B)$'s.) 

\begin{remark}
In \cite{Maier} $\nu = C/c_n$ is used instead. Also $l_0(B)$, $l_i(B)$, $lt_0(B)$, $lt_i(B)$, and $l_\theta(B)$ ($i = 1, 2, 3$) are written there as $L^I_\ell(B; g_2, g_3)$, $L^{II}_\ell(B; e_i, g_2, g_3)$, $Lt^I_\ell(B; g_2, g_3)$, $Lt^{II}_\ell(B; e_i, g_2, g_3)$, and $L\theta_\ell(B; g_2, g_3)$ respectively, where $\ell = n$.  
\end{remark}

The compatibility equations from the recursive formulas for these special cases give rise to explicit formulas for $lt_j(B)$'s and $l_\theta(B)$'s. Tables for $lt_0(B)$, $l_\theta(B)$ up to $n = 8$, and for $lt_i(B)$ up to $n = 6$, are given in \cite[Table 5, 6]{Maier}. 

\begin{example} \label{lt-poly}
We recall Maier's formulas for $lt_j(B)$ and $l_\theta(B)$ for $n \le 4$.

(1) First of all, $l_\theta(B) = 1$ for $n \le 3$. For $n = 4$, 
$$
l_\theta(B) = B^2 - \tfrac{193}{3} g_2.
$$ 
Also for $n = 1$, $lt_j(B) = 1$ for all $j$.

(2) $n = 2$: $lt_0(B) = 1$, $lt_i(B) = B - 6e_i$ for $i = 1, 2, 3$.

(3) $n = 3$: $lt_0(B) = B^2 - \tfrac{75}{4} g_2$, and for $i  = 1, 2, 3$,
$$
lt_i(B) = B^2 - 15 e_i B + \tfrac{75}{4} g_2 - 225 e_i^2.
$$

(4) $n = 4$: $lt_0(B) = B^3 - \tfrac{343}{4} g_2 B - \tfrac{1715}{2} g_3$. For $i = 1, 2, 3$, 
\begin{equation*}
\begin{split}
lt_i(B) &= B^4 - 55 e_i B^3 + (\tfrac{539}{4} g_2 - 945 e_i^2) B^2 \\
&\qquad + (1960 e_i g_2 + 2450 g_3) B + 61740 e_i^2 g_2 - 68600 e_i g_3 - 9261 g_2^2.
\end{split}
\end{equation*} 
\end{example}

To apply Theorem \ref{xyk}, we need to compare the projection map 
\begin{equation} \label{e:pi}
\pi_n: \bar Y_n \to E, \qquad a \mapsto \pi_n(a) := a_0.
\end{equation}
with the addition map $\sigma_n: \bar Y_n \to E$. They turn out to be the same!

\begin{theorem} \label{pi=sigma}
$\pi_n(a) = \sigma_n(a)$. Moreover, $\kappa(a) = -\z_n(a)$.
\end{theorem}

\begin{proof}
During the proof we view $a_i \in \Bbb C$ instead of its image $[a_i] \in E$.

Let $a \in Y_n$. The two expressions (\ref{ansatz}) and (\ref{HK}), which correspond to the same solution to the Lam\'e equation (\ref{lame}), must be proportional to each other by a constant. Hence we get
$$
\kappa(a) = \sum_{i = 1}^n \zeta(a_i) - \zeta(a_0).
$$
Recall that $\z_n(a) = \zeta(\sigma_n(a)) - \sum_{i = 1}^n \zeta(a_i)$. Then
\begin{equation} \label{z+k}
\z_n(a) + \kappa(a) = \zeta(\sigma_n(a)) - \zeta(a_0).
\end{equation}
As a well defined meromorphic function on $\bar Y_n$, we conclude that
$$
a_0(a) = \sigma_n(a) + c
$$
for some constant $c \in \Bbb C$. Consider a point $a \in Y_n \setminus X_n$ with $\sigma_n(a) = \tfrac{1}{2} \omega_1$, i.e.~$l_1(B_a) = 0$. Such $a$ exists by Proposition \ref{4lame}. Then $\z_n(a) = 0$ trivially. We also have $\kappa(a) = 0$ by Theorem \ref{xyk} since 
$$
C_a^2 = c_n^2 l_0(B_a) l_1(B_a) l_2(B_a) l_3(B_a) = 0
$$ 
(again by Proposition \ref{4lame}). So (\ref{z+k}) implies $0 = \tfrac{1}{2}\eta_1 - \zeta(\tfrac{1}{2} \omega_1 + c)$, and hence $c = 0$. This proves $\sigma_n(a) = a_0$, which represents $\pi_n(a)$ in $E$, and also $\kappa(a) = -\z_n(a)$. The proof is complete.
\end{proof}

Now we may describe the explicit construction of the polynomial $W_n(\bz)$ in Theorem \ref{main-thm} based on Theorem \ref{t:xyk}. It is indeed merely an application of the elimination theory using resultant.

By Theorem \ref{t:xyk} and \ref{pi=sigma}, we may eliminate $C$ to get
\begin{equation} \label{y/z}
\frac{y_0}{\z_n} = \frac{16}{n^2(n + 1)^2(n - 1)(n + 2)} \frac{lt_1(B) lt_2(B) lt_3(B)}{l_0(B) lt_0(B)^2 l_\theta(B)},
\end{equation}
which leads to a polynomial equation $g = 0$ for
\begin{equation} \label{g=0}
g := \bz \prod_{i = 1}^3 lt_i(B) - y_0 \frac{n^2(n + 1)^2(n - 1)(n + 2)}{16} l_0(B) lt_0(B)^2 l_\theta(B). 
\end{equation}
On the other hand, the three rational expressions of $x_0$ lead to $f = 0$ for
\begin{equation} \label{f=0}
\begin{split}
f &:= l_i(B) lt_i(B)^2 - (x_0 - e_i) \frac{n^2(n + 1)^2}{4} l_0(B) lt_0(B)^2 \\
&= \frac{1}{3}\sum_{i = 1}^3 l_i(B) lt_i(B)^2 - x_0 \frac{n^2(n + 1)^2}{4} l_0(B) lt_0(B)^2.
\end{split}
\end{equation}
Notice that $f, g$ are polynomials in $g_2, g_3$ (and $B, x_0, y_0$) instead of $e_i$'s.

Let $R(f, g; B)$ be the \emph{resultant of the two polynomials} $f$ and $g$ arising from the elimination of the variable $B$. Standard elimination theory (see e.g~\cite[Chapter 5]{Hassett}) implies that $R(f, g; B)$ gives \emph{the equation} defining the branched covering map $\sigma_n: \bar Y_n \to E$ outside the loci $C = 0$:

\begin{proposition} \label{p:resultant}
$R(f, g; B)(\bz) = \lambda_n W_n(\bz) \in \Bbb Q[g_2, g_3, x_0, y_0][\bz]$, where $\lambda_n = \lambda_n(g_2, g_3, x_0, y_0)$ is independent of $\bz$. 
\end{proposition}

In particular, the pre-modular form $Z_n(\sigma; \tau) = W_n(Z)(\sigma; \tau)$ can be explicitly computed for any $n \in \Bbb N$ by way of the resultant $R(f, g; B)$. 

In practice, such a computation is time consuming even using computer. In the following, we apply it to the initial cases up to $n = 4$. As before we denote $x_0 = \wp(\sigma) =: \wp$ and $y_0 = \wp'(\sigma) =: \wp'$.

\begin{example}
For $n = 2$, it is easy to see that \small
\begin{equation*}
\begin{split}
f &= B^3 - 9\wp B^2 + 27(g_2 \wp + g_3), \\
g &= \bz B^3 - 9\wp' B^2 - 9\bz g_2 B + 27 (g_2 \wp' - 2 \bz g_3).
\end{split}
\end{equation*} \normalsize
The resultant $R(f, g; B)$ is calculated by the $6 \times 6$ \emph{Sylvester determinant}:
\small
\begin{equation*} 
\begin{vmatrix}
1 & -9\wp & 0 & 27(g_2 \wp + g_3) & 0 & 0 \\
0 & 1 & -9\wp & 0 & 27(g_2 \wp + g_3) & 0 \\
0 & 0 & 1 & -9\wp & 0 & 27(g_2 \wp + g_3) \\
\bz & -9\wp' & -9 \bz g_2 & 27(g_2 \wp' - 2\bz g_3) & 0 & 0 \\
0 & \bz & -9\wp' & -9 \bz g_2 & 27(g_2 \wp' - 2\bz g_3) & 0 \\
0 & 0 & \bz & -9\wp' & -9 \bz g_2 & 27(g_2 \wp' - 2\bz g_3) 
\end{vmatrix}.
\end{equation*}
\normalsize  
A direct evaluation gives
$$
R(f, g; B)(\bz) = -3^9 \Delta (\wp')^2 (\bz^3 - 3\wp \bz - \wp').
$$
Here $\Delta = g_2^3 - 27 g_3^2$ is the discriminant. This gives $W_2(\bz) = \bz^3 - 3\wp \bz - \wp'$ and $Z_2(\sigma; \tau) = W_2(Z) = Z^3 - 3\wp Z - \wp'$.
\end{example}

\begin{example} 
For $n = 3$, we have \small
\begin{equation*}
\begin{split}
f &= 16 B^6 - 576 B^5 \wp + 360 B^4 g_2 + 5400 B^3 (5 g_3 + 4 g_2 \wp) \\
&\qquad- 3375 B^2 g_2^2 - 84375 \Delta  - 101250 B g_2 (3 g_3 + 2 g_2 \wp), \\
g &= 16 B^6 \bz -1440 B^5 \wp' - 1800 B^4 g_2 \bz + 54000 B^3 (g_2 \wp' - g_3 \bz) \\
&\qquad - 16875 B^2 g_2^2 \bz - 506250 B g_2^2 \wp' + 421875 \Delta \bz.
\end{split}
\end{equation*} \normalsize

It takes a couple seconds to evaluate the corresponding $12 \times 12$ Sylvester determinant (e.g.~using \emph{Mathematica}) to get
$$
R(f, g; B)(\bz) = 2^{36} 3^{27} 5^{30} \Delta^5 (\wp')^4 W_3(\bz),
$$
where $W_3(\bz)$ is given by
$$
W_3(\bz) = \bz^6 - 15 \wp \bz^4 - 20 \wp' \bz^3 +(\tfrac{27}{4} g_2 - 45 \wp^2) \bz^2 - 12 \wp \wp' \bz - \tfrac{5}{4} \wp'^2. 
$$
It seems impractical to evaluate this resultant by hand. .
\end{example}

Both $Z_2$ and $Z_3$ are known to Dahmen \cite{Dahmen}. Here is a new example:

\begin{example} \label{ex:W4}
For $n = 4$, the expansions of the polynomials $f$ and $g$, as given in (\ref{f=0}) and (\ref{g=0}) by a direct substitution, are already too complicate to put here. Nevertheless, a couple hours \emph{Mathematica} calculation gives
$$
R(f, g; B)(\bz) = -2^{80} 3^{63} 5^{60} 7^{63} \Delta^{18} (\wp')^8 W_4(\bz),
$$
where $W_4(\bz)$ is the degree 10 polynomial:\small
\begin{equation} \label{W4} 
\begin{split}
W_4(\bz) &= \bz^{10} - 45 \wp \bz^8 - 120 \wp' \bz^7 + (\tfrac{399}{4}g_2 - 630 \wp^2) \bz^6 -504 \wp \wp' \bz^5 \\
&\quad  - \tfrac{15}{4} (280 \wp^3 - 49 g_2 \wp - 115 g_3) \bz^4 + 15(11 g_2 - 24 \wp^2) \wp' \bz^3\\
&\qquad  - \tfrac{9}{4} (140 \wp^4 - 245 g_2 \wp^2 + 190 g_3 \wp + 21 g_2^2) \bz^2 \\
&\qquad \quad -(40 \wp^3 - 163 g_2 \wp + 125 g_3) \wp' \bz + \tfrac{3}{4}(25 g_2 - 3\wp^2) (\wp')^2.
\end{split} 
\end{equation} \normalsize
The weight 10 pre-modular form $Z_4(\sigma; \tau)$ is then obtained.
\end{example}

We end this section with a brief discussion on the \emph{rationality property}. We have constructed two affine curves from $\bar X_n$. One is the hyperelliptic model $Y_n = \{(B, C)\mid C^2 = \ell_n(B)\}$, another one is $Y_n' : = \{(x_0, y_0, \bz) \mid y_0^2 = 4x_0^2 - g_2 x_0 - g_3, \, W_n(x_0, y_0; \bz) = 0\}$ which is understood as a degree $\tfrac{1}{2} n(n + 1)$ branched cover of the original curve $E = \{(x_0, y_0) \mid y_0^2 = 4 x_0^3 - g_2 x_0 - g_3\}$ under the projection $\sigma_n' : Y'_n \to E$ with defining equation $W_n(\bz) = 0$. 

$Y_n$ is birational to $Y'_n$ over $E$, namely the addition map $\sigma_n: Y_n \to E$ is compatible with $\sigma_n': Y_n' \to E$. Notice that both $\ell_n$ and $W_n$ have coefficients in $\Bbb Q[g_2, g_3]$. The explicit birational map $\phi: (B, C) \dasharrow (x_0, y_0, \bz)$ (given in Theorem \ref{t:xyk} and \ref{pi=sigma} via $\z_n = -\kappa$) also has coefficients in $\Bbb Q[g_2, g_3]$. This implies that $\phi$ is defined over $\Bbb Q$. Moreover $\phi$ extends to a \emph{birational morphism} 
$$
\xymatrix{\bar Y_n \cong \bar X_n \ar[rr]^\phi \ar[rd]^{\sigma_n} & &\bar Y_n' \ar[ld]_{\sigma_n'} \\ & E}
$$ 
by identifying $\sigma_n^{-1}(0_E)$ with $\z_n^{-1}(\infty)$. The morphism $\phi$ is an isomorphism outside those branch points for $Y_n \to \Bbb P^1$ (i.e.~$C = 0$). In particular, the non-isomorphic loci lie in $\z_n = 0$ by \eqref{xyk} and Theorem \ref{pi=sigma}. 

\begin{remark}
In contrast to the smoothness of $Y_n(\tau)$ for general $\tau$, for all $n \ge 3$ the model $Y'_n(\tau)$ is singular at points $\bz = 0 = y_0$ (and hence $x_0 = e_i$ for some $i$). Indeed from \eqref{xyk} this is equivalent to $C = 0$ and $l_i(B) lt_i(B)^2 = 0$ for some $1 \le i \le 3$. For $n = 2$, there is only one solution $B$ for each fixed $i$ (c.f.~Example \ref{lt-poly}). However, for $n \ge 3$ there are more than one solutions $B$. These points $(B, 0) \in Y_n$ are collapsed to the same point $(x_0, y_0, \bz) = (e_i, 0, 0) \in Y'_n$ under $\phi$, thus $(e_i, 0, 0)$ is a singular point of $Y'_n$. 

For $n = 3, 4$ this is easily seen from the equation $W_n(\bz) = 0$ given above since it contains a quadratic polynomial in $(\bz, \wp')$ as its lowest degree terms.
\end{remark}

In particular, the birational map $\phi^{-1}$ is also represented by rational functions $B = B(x_0, y_0, \bz)$ and $C = C(x_0, y_0, \bz)$ with coefficients in $\Bbb Q[g_2, g_3]$ and with at most poles along $\bz = 0$. In principle such an explicit inverse can be obtained by a Groebner basis calculation associated to the ideal of the graph $\Gamma_\phi$. 
%
%
The following statement is clear from the above description:

\begin{proposition}
Let $E$ be defined over $\Bbb Q$, i.e.~$g_2, g_3 \in \Bbb Q$. Then the Lam\'e curve $\bar Y_n$ is also defined over $\Bbb Q$ for all $n \in \Bbb N$. Moreover, $\bar Y'_n$ and all the morphisms $\sigma_n, \sigma_n', \phi$ are also defined over $\Bbb Q$. 

A rational point $(B, C) \in \bar Y_n$ is mapped to a rational point $(x_0, y_0, \bz) \in \bar Y_n'$ by $\phi$. For the converse, given $(x_0, y_0) \in E(\Bbb Q)$, a point $(x_0, y_0, \bz)$ in the $\sigma_n'$ fiber gives a unique $(B, C) \in \bar Y_n(\Bbb Q)$ if $\bz \in \Bbb Q$ and $(x_0, y_0, \bz) \ne (e_i, 0, 0)$ for any $i$.
\end{proposition}

\begin{remark}
It is well known that there are only few (i.e.~at most finite) rational points on a \emph{non-elliptic} hyperelliptic curve. This phenomenon is consistent with the irreducibility of the polynomial $W_n(\bz)$ over $K(E)$ in light of Hilbert's irreducibility theorem that there is a infinite (Zariski dense) set of $(g_2, g_3, x_0, y_0) \in \Bbb Q^4$ so that the specialization of $W_n(\bz)$ is still irreducible. 
Nevertheless, it might be interesting to see if $\z_n$ plays any role in the study of rational points. 
\end{remark}

\appendix

\section{A counting formula for Lam\'e equations} \label{Chou}

\centerline{By You-Cheng Chou \footnote{Taida Institute for Mathematical Sciences (TIMS), National Taiwan University, Taipei, Taiwan. Email: b99201040@ntu.edu.tw}}

\bigskip

Using the pre-modular forms constructed in \S \ref{new-Zn} and \S \ref{resultant}, we verify the $n = 4$ case of Dahmen's conjectural counting formula (Conjecture 73 in \cite{Dahmen}) for integral Lam\'e equations with finite monodromy. It is known that the finite monodromy group is necessarily a dihedral group. 

\subsection{Dahmen's conjecture}

Let $L_n(N)$ be the number of Lam\'e equations $w'' = (n(n + 1) \wp(z) + B) w$ up to linear equivalence which has finite monodromy isomorphic to the dihedral group $D_N$. Using the Hermite--Halphen ansatz \eqref{ansatz} and the theory in \S \ref{new-Zn}, the problem is reduced to the zero counting of the ${\rm SL}(2, \Bbb Z)$ modular form
\[
M_n(N):= \prod_{\substack{0\leq k_1,k_2 < N\\{\rm gcd}(k_1,k_2,N)=1}} Z_n\Big(\frac{k_1+k_2\tau}{N}; \tau\Big).
\]
Using this, by repeating Dahmen's argument in \cite{Dahmen}, Lemma 65, 74, we get

\begin{proposition} \label{mono-proj}
Suppose that for all $N\in\mathbb{Z}_{\geq 3}$ and $n\in\mathbb{N}$ we have that 
\[
\nu_{\infty}(M_n(N)) = a_n\phi(N) + b_n\phi\Big(\frac{N}{2}\Big),
\]
where 
$a_{2m}=a_{2m+1}=m(m+1)/2$, $b_{2m}=b_{2m-1}=m^2$. Then
\[
L_n(N) =\tfrac{1}{2} \left( \frac{n (n + 1) \Psi(N)}{24}-\left( a_n \phi(N) + b_n\phi \Big(\frac{N}{2}\Big) \right) \right) + \tfrac{2}{3}\epsilon_n(N),
\]
where $\epsilon_n(N) = 1$ if $N = 3$ and $n \equiv 1 \pmod 3$, and $\epsilon_n(N) = 0$ otherwise.

Furthermore, $Z_n(\sigma;\tau)$ with $\sigma$ a torsion point has only simple zeros in $\tau\in\mathbb{H}$.
\end{proposition}

\begin{proof}
Recall the formula for ${\rm SL}(2, \Bbb Z)$ modular forms of weight $k$: 
\[
\sum_{P\neq \infty,\, i,\, \rho}\nu_P(f)+\nu_{\infty}(f)+\frac{\nu_i(f)}{2}+\frac{\nu_{\rho}(f)}{3}=\frac{k}{12}.
\]
For $f = M_n(N)$, the weight $k = \frac{1}{2} n(n + 1) \Psi(N)$. Notice that the counting is always doubled under the symmetry $(k_1, k_2) \to (N - k_1, N - k_2)$, thus by \cite{Dahmen}, Lemma 65, an upper bound for $L_n(N)$ is given by 
\[
U_n(N) := \tfrac{1}{2} \left( \frac{n (n + 1) \Psi(N)}{24}-\left( a_n \phi(N) + b_n\phi \Big(\frac{N}{2}\Big) \right) \right) + \tfrac{2}{3}\epsilon_n(N).
\] 
That is, $L_n(N)\leq U_n(N)$. Moreover, the equality holds if and only if each factor $Z_n((k_1 + k_2\tau)/N; \tau)$ has only simple zeros. 

We will show the equality holds by comparing it with the counting formula for the projective monodormy group $PL_n(N)$ (c.f.~\cite{Dahmen}, Lemma 74).

We recall the relation between $L_n(N)$ and $PL_n(N)$:
\[
PL_n(N)=\left\lbrace 
\begin{array}{ll}
L_n(N)+L_n(2N) & {\rm if}\ N\ {\rm is\ odd}, \\
L_n(2N) & {\rm if}\ N\ {\rm is\ even}.
\end{array}
\right.
\]
If $n$ is even and $N$ is odd, we have
\[
\begin{split}
&PL_n(N) = L_n(N) + L_n(2N)\\ 
&\leq \tfrac{1}{2} \left( \frac{n(n + 1) \Psi(N)}{24}-\left( \frac{\frac{n}{2}(\frac{n}{2}+1)}{2}\phi(N) + \frac{n^2}{4}\phi\Big(\frac{N}{2}\Big) \right) \right) + \tfrac{2}{3}\epsilon_n(N)
\\
& \quad + \tfrac{1}{2} \left( \frac{n(n + 1) \Psi(2N)}{24}-\left( \frac{\frac{n}{2}(\frac{n}{2}+1)}{2}\phi(2N) + \frac{n^2}{4}\phi(N) \right) \right) + \tfrac{2}{3}\epsilon_n(2N)
\\
& = \frac{n(n+1)}{12}\left(\Psi(N)-3\phi(N)\right)+\tfrac{2}{3}\epsilon_n(N)
\end{split}
\]
For the last equality, we use $\epsilon_n(2N)=0$, $\Psi(2N) = 3\Psi(N)$ and $\phi(2N) = \phi(N)$. (If $N$ is even, the relations are $\epsilon_n(N) = 0$, $\Psi(2N) = 4\Psi(N)$ and $\phi(2N)=\phi(N)$.)
For the other three cases with $(n, N)$ being (even, even), (odd, odd) or (odd, even), the computations are similar, and all lead to
\[
PL_n(N)\leq \frac{n(n+1)}{12}\left(\Psi(N)-3\phi(N)\right)+\tfrac{2}{3}\epsilon_n(N).
\]

On the other hand, using the method of \emph{dessin d'enfants}, Dahmen showed directly that the equality holds \cite{Dahmen2}. Thus all the intermediate inequalities are indeed equalities, and in particular $L_n(N) = U_n(N)$ holds.
\end{proof}

\subsection{$q$-expansions for some modular forms} \label{q-expansions}
Recall that
\[
\begin{split}
\sum_{m\in\mathbb{Z}} \frac{1}{(m+z)^k} &= \frac{1}{(k-1)!}(-2\pi i)^k \sum_{n=1}^{\infty} n^{k-1} e^{2\pi i n z}, \\
\sum_{n\in\mathbb{Z}} \frac{1}{(x+n)^2} &= \pi^2 \cot^2(\pi x)+\pi^2, \\
\sum_{n\in\mathbb{Z}} \frac{1}{(x+n)^3} &= \pi^3 \cot^3(\pi x)+\pi^3 \cot(\pi x).
\end{split}
\]
We compute the $q$-expansions for $g_2,g_3,\wp,\wp',Z$, where $q=e^{2\pi i\tau}$:
\[
\begin{split}
g_2 = 60 \sum_{(n,m)\neq (0,0)} \frac{1}{(n+m\tau)^4} = 60\Big( 2\zeta(4)+2\frac{(-2\pi i)^4}{3!}\sum_{n=1}^{\infty}\sigma_3(n) q^{n}\Big),
\end{split}
\] where $\sigma_k(n):=\sum_{d|n}d^k$. Similarly, 
\[
\begin{split}
g_3 = 140 \sum_{(n,m)\neq (0,0)} \frac{1}{(n+m\tau)^6} = 140\Big( 2\zeta(6)+2\frac{(-2\pi i)^6}{5!}\sum_{n=1}^{\infty}\sigma_5(n) q^{n}\Big).
\end{split}
\]

Let $z=t+s\tau$. For $s=0$, we have \small
\[
\begin{split}
\wp'(t; \tau) 
&= -2\sum_{n,m\in\mathbb{Z}}\frac{1}{(t+n+m\tau)^3}
\\
&= -2\sum_{n\in\mathbb{Z}}\frac{1}{(t+n)^3} -2\sum_{m=1}^{\infty}\sum_{n\in\mathbb{Z}}\left( \frac{1}{(m\tau+n+t)^3}-\frac{1}{(m\tau+n-t)^3}\right) 
\\
&= -2\sum_{n\in\mathbb{Z}}\frac{1}{(t+n)^3} - 2\sum_{m=1}^{\infty} \frac{(-2\pi i)^3}{2!}\sum_{n=1}^{\infty} n^2\left( e^{2\pi in(m\tau +t)} - e^{2\pi in(m\tau-t)}\right)
\\
&= -2\pi^3 \cot(\pi t) - 2\pi^3 \cot^3(\pi t) + 16\pi^3 \sum_{n,m=1}^{\infty} n^2 \sin(2\pi nt)\,q^{nm}.
\end{split}
\] 
\[
\begin{split}
&\wp(t; \tau) = \frac{1}{t^2}+\sum_{(n,m)\neq (0,0)}\left( \frac{1}{(t+n+m\tau)^2}-\frac{1}{(n+m\tau)^2}\right)
\\
&=\sum_{n\in\mathbb{Z}}\frac{1}{(t+n)^2}-\sum_{n=1}^{\infty}\frac{2}{n^2} + \sum_{m=1}^{\infty}\sum_{n\in\mathbb{Z}}\left( \frac{1}{(m\tau+t+n)^2}+\frac{1}{(m\tau-t+n)^2}-\frac{2}{(m\tau+n)^2}\right)
\\
&= \pi^2 \cot^2(\pi t)+\tfrac{2}{3}\pi^2+\sum_{m=1}^{\infty}(-2\pi i)^2 \sum_{n=1}^{\infty}\left( e^{2\pi in(m\tau+t)} + e^{2\pi in(m\tau-t)} - 2e^{2\pi inm\tau}\right)
\\
&= \pi^2 \cot^2(\pi t)+\tfrac{2}{3}\pi^2 +8\pi^2\sum_{n, m=1}^{\infty}(1-\cos 2n\pi t)q^{nm}.
\end{split}
\]
\normalsize
Also, the Hecke function $Z$ (cf.~\eqref{e:Hecke}): 
$$
Z(t; \tau) = \pi \cot(\pi t)+4\pi \sum_{n, m=1}^{\infty} (\sin 2n\pi t) q^{nm}.
$$

For $s=\tfrac{1}{2}$, we have \small
\[
\begin{split}
&\wp'(t + \tfrac{1}{2}\tau; \tau) = -2\sum_{(n, m)\neq (0,0)}\frac{1}{(t+n+(\frac{1}{2}+m)\tau)^3}
\\
&= -2\sum_{m=1}^{\infty}\left( \sum_{m\in\mathbb{Z}}\frac{1}{(n+t+(m-\frac{1}{2})\tau)^3}-\sum_{n\in\mathbb{Z}}\frac{1}{(n-t+(m-\frac{1}{2})\tau)^3}\right)
\\
&= -2\frac{(-2\pi i)^3}{2!}\sum_{n, m=1}^{\infty}n^2\left( e^{2\pi in(t+(m-\frac{1}{2})\tau)} - e^{2\pi in(-t)+(m-\frac{1}{2})\tau}\right)
\\
&= 16\pi^3 \sum_{n, m=1}^{\infty}n^2(\sin 2\pi nt)q^{n(m-\frac{1}{2})}.
\end{split}
\] \normalsize
Similarly, 
\[
\wp(t + \tfrac{1}{2}\tau; \tau) = -\tfrac{1}{3}\pi^2 +8\pi^2 \sum_{n,m=1}^{\infty} nq^{nm}-8\pi^2\sum_{n,m=1}^{\infty} n(\cos 2\pi nt)q^{n(m-\frac{1}{2})},
\] 
and $Z(t + \tfrac{1}{2}\tau; \tau) = 4\pi\sum_{n,m=1}^{\infty}(\sin 2\pi nt) q^{n(m-\frac{1}{2})}$.

\subsection{The counting formula for $n=4$}

Now we give the computations for $n=4$ and prove the formula $L_4(N) = U_4(N)$ from Proposition \ref{mono-proj}.

\begin{theorem}\label{count n=4}
For $n=4$ and $N\in\mathbb{Z}_{\geq 3}$, we have 
\[
L_4(N)=\tfrac{1}{2}\left( \tfrac{5}{6}\Psi(N) - \left( 3\phi(N)+4\phi\Big(\frac{N}{2}\Big)\right)\right).
\] 
Moreover, $Z_4(\sigma; \tau)$ with $\sigma \in E_\tau[N]$ has only simple zeros in $\tau \in \Bbb H$.
\end{theorem}

\begin{proof}
For $n=4$, the pre-modular form $Z_4 = W_4(Z)$ is given in \eqref{W4}: \small
\[
\begin{split}
W_4(Z) &=Z^{10} - 45 \wp Z^8 - 
120 \wp' Z^7 + (\tfrac{399}{4} g_2 - 630 \wp^2) Z^6 - (504 \wp \wp') Z^5 
\\
&\qquad - \tfrac{15}{4} (280 \wp^3 - 49 g_2 \wp - 115 g_3) Z^4 + 15 (11 g_2 - 24 \wp^2) \wp' Z^3 
\\ 
&\qquad\qquad - \tfrac{9}{4} (140 \wp^4 - 245 g_2 \wp^2 + 190 g_3 \wp + 21 g_2^2) Z^2 
\\ 
&\qquad\qquad\qquad - (40 \wp^3 - 163 g_2 \wp + 125 g_3) \wp Z + \tfrac{3}{4} (25 g_2 - 3 \wp^2) \wp'^2,
\end{split}
\] \normalsize
where $Z$ is the Hecke function. We compute the asymptotic behavior of $W_4(Z)$ when $\tau\rightarrow \infty$. Let $z=t+s\tau$. We divide the problem into two cases

(1) $s\equiv 0 \pmod{1}$: According to the $q$-expansion given in \S \ref{q-expansions}, we have
\[
g_2\rightarrow \tfrac{3}{4}\pi^4,\qquad g_3\rightarrow \tfrac{8}{27}\pi^6,\qquad  Z(z)\rightarrow \pi \cot(\pi t),
\]
\[
\begin{split}
\wp'(z) \rightarrow -2\pi^3 \cot(\pi t)-2\pi^3 \cot^3(\pi t), \qquad
\wp(z) \rightarrow \pi^2 \cot^2(\pi t)+\tfrac{2}{3}\pi^2.
\end{split}
\]
A direct computation shows that $W_4(Z)$ has zeros at $\infty$ when $s=0$. 

By replacing all the modular forms $g_2,g_3,\wp,\wp'$ and $Z$ in $W_4(Z)$ with their $q$-expansions, we have (e.g.~using \emph{Mathematica})
\[
W_4(Z)= 2^{14} 3^3 5^2 7 \,\pi^{10} \cos^2(\pi t)\sin^2 (\pi t) q^3 + O(q^4)
\]

(2) $s\not\equiv 0 \pmod{1}$: In this case we have
\[
Z\rightarrow 2\pi i\left( s-\tfrac{1}{2}\right),\qquad  \wp(z)\rightarrow -\tfrac{1}{3}\pi^2,
\]
\[
\wp'(z)\rightarrow 0,\qquad g_2\rightarrow \tfrac{4}{3}\pi^4,\qquad g_3\rightarrow \tfrac{8}{27}\pi^6.
\]
Hence the constant term of $W_4(Z)$ is given by 
\[
\begin{split}
W_4(z) &= - 64 \pi^{10} (-2 + s)  (-1 + s)^2 s^2  (1 + s) \\
& \qquad \times (-3 + 2 s) (-1 + 2 s)^2 (1 + 2 s) + O(q).
\end{split}
\]
If $s\not\equiv 0 \pmod{1}$ then $W_4(Z)$ has zero at $\tau = \infty \Longleftrightarrow s\equiv \tfrac{1}{2} \pmod{1}$.

Now we fix $s=\tfrac{1}{2}$ and replace the modular forms $g_2,g_3,\wp,\wp'$ and $Z$ in $W_4(Z)$ with their $q$-expansions. We get 
\[
W_4(Z)= 2^{10} 3^3 5^2 7\, \pi^{10} \cos(\pi t)^2 \sin(\pi t)^2 q^2 + O(q^3).
\]
These computations for the $q$-expansions imply that
\[
\begin{split}
\nu_{\infty}(M_4(N)) &= 3 \, \#\left\lbrace\, 1 \leq k_1 \leq N\mid {\rm gcd}(N, k_1) = 1\,\right\rbrace 
\\
&\qquad + 2 \, \#\left\lbrace\, 0\leq k_1\leq N \mid {\rm gcd}(N/2, k_1) = 1\, \right\rbrace
\\
& = 3\phi(N) + 4\phi(N/2).
\end{split}
\]
Since the value of $\nu_{\infty}(M_4(N))$ coincides with the assumption in Proposition \ref{mono-proj} for $n = 4$, the theorem follows from it accordingly.
\end{proof}

\end{document}